\documentclass[12 pt, reqno]{amsart}
\usepackage{amsmath, amscd, graphicx, latexsym, times, }
\usepackage{amsmath,amssymb,amsthm,graphicx}
\usepackage[usenames]{color}
\usepackage[shortlabels]{enumitem}
\usepackage{hyperref}
\usepackage{xparse}
\usepackage{float}
\usepackage{morefloats}
\usepackage{lipsum}
\usepackage{mathtools}
\usepackage{xcolor}

\usepackage{xparse}
\usepackage{float}
\usepackage{mathtools}

\newtheorem{theorem}{Theorem}[section]
\newtheorem*{theorem-non}{Theorem}
\newtheorem{prop}[theorem]{Proposition}
\newtheorem{lemma}[theorem]{Lemma}
\newtheorem{remark}[theorem]{Remark}
\newtheorem{question}[theorem]{Question}
\newtheorem{definition}[theorem]{Definition}

\newcommand\Z{{\mathbb Z}}

\newcommand{\CPb}{\overline{\mathbb{CP}}^{2}}
\newcommand{\CP}{{\mathbb{CP}}^{2}}

\def \CPb {\overline{\mathbb{CP}}^{2}}
\def \CP {{\mathbb{CP}}^{2}} 

\def \Z {\mathbb{Z}}

\def \- {\setminus}

\def \Q {\mathbb{Q}}

\makeatletter
\@namedef{subjclassname@2020}{%
  \textup{2020} Mathematics Subject Classification}
\makeatother

\title[Symplectic 4-Manifolds on the Noether line and between Noether and half Noether lines]{Symplectic 4-manifolds on the Noether line and between the Noether and half Noether lines}

\begin{document}
\author{S\"{u}meyra Sakalli}
\address{Department of Mathematical Sciences,
University of Arkansas,
Fayetteville, AR, 72701, USA}
\email{ssakalli@uark.edu}



\maketitle

\subsection*{Abstract}

We construct simply connected, minimal, symplectic 4-manifolds with exotic smooth structures and each with one Seiberg-Witten basic class up to sign, on the Noether line and between the Noether and half Noether lines by star surgeries introduced by Karakurt and Starkston, and by using complex singularities. We also construct certain configurations of complex singularities in the rational elliptic surfaces geometrically, without using any monodromy arguments. By using these configurations, we give symplectic embeddings of star shaped plumbings inside (some blow-ups of) elliptic surfaces.

\section{Introduction}

\label{intro}
For a closed, simply connected, symplectic $4$-manifold $X$, a pair of invariants are defined as follows: $\chi_h(X) := (e(X) + \sigma(X))/4$ and $c_{1}^{2} (X) := 2e(X) + 3\sigma(X)$, where $e(X)$ and $\sigma(X)$ denote the Euler characteristic and the signature of $X$, respectively. The $(\chi_h, c_1^2)$-plane is called the geography chart on which the following lines 

\begin{equation}
c_1^2 = 2\chi_h - 6 \;\;\; \text{and}\;\;\; c_1^2 = \chi_h - 3
\end{equation}

\noindent are called the Noether and half Noether lines, respectively. Note that for minimal complex surfaces $S$ of general type, the Noether inequality $c_1^2(S) \geq 2\chi_h (S) - 6$ holds (see e.g. \cite{BHPV}). Moreover, it is known that all minimal complex surfaces of general type have exactly one (Seiberg-Witten) basic class, up to sign \cite{W}. Thus, it is natural to ask if one can construct smooth 4-manifolds with one basic class. 
In \cite{FS1}, Fintushel and Stern built a family of simply connected, spin, smooth, nonsymplectic 4-manifolds with one basic class. Then, Fintushel, Park and Stern constructed a family of simply connected, noncomplex, {\em symplectic} 4-manifolds with one basic class which fill the region between the half-Noether and Noether lines in the $(\chi_h, c_1^2)$-chart \cite{FPS}. Later Akhmedov constructed infinitely many simply connected, nonsymplectic and pairwise nondiffeomorphic 4-manifolds with nontrivial Seiberg-Witten invariants \cite{A}. Park and Yun also gave a construction of an infinite family of simply connected, nonspin, smooth, nonsymplectic 4-manifolds with one basic class \cite{PY}. All these manifolds were obtained via knot surgeries, blow-ups and rational blow-downs. 

In \cite{KS}, Karakurt and Starkston introduced {\em star surgeries} which are new 4–dimensional symplectic operations. A star surgery is the operation of cutting out the neighborhood of a star shaped plumbing of symplectic 2-spheres inside a symplectic 4-manifold, and replacing it with a convex symplectic filling of strictly smaller Euler characteristic. 
Also in \cite{Star2}, Starkston showed that infinitely many star surgeries are not equivalent to any sequences of generalized symplectic rational blow-downs.

In this paper we give new constructions of simply connected, minimal and symplectic 4-manifolds on the Noether line and between the Noether and half Noether lines by using various types of star surgeries and complex singularities. We also show that each of our manifolds has one Seiberg-Witten basic class up to sign, and an exotic smooth structure. By the latter we mean that they are homeomorphic but not diffeomorphic to the manifolds with standard smooth structures. We would like to note that in \cite{Ham}, symplectic 4-manifolds on the Noether line with bigger Euler characteristics were built. 
On the other hand, in the literature there are different constructions of symplectic 4-manifolds having the same topological invariants (e.g. \cite{FPS, ABKP}). However, here we give a completely different construction and we do not know if our manifolds are diffeomorphic to the previously constructed ones. In fact, giving different constructions of smooth or symplectic 4-manifolds with the same invariants is interesting and an active research area. For instance, see \cite{P2, SS1, FS3, Mich, KS, AM, AS2} for distinct constructions of symplectic 4-manifolds which are all exotic copies of $\CP \# 6\CPb$ and $\CP \# 7\CPb$. However, as of today, it is not known how to distinguish the smooth structures of symplectic, exotic 4-manifolds that have the same topological invariants but are obtained in different ways, and it is an intriguing question.

Let us give the outline of the paper. In sections \ref{SW}, and \ref{recap} we give brief background on the Seiberg-Witten invariants and four types of star surgeries of \cite{KS}, respectively. In Section \ref{2nd}, we construct three different configurations of $I_n$ singularities in the rational elliptic surface $E(1) := \CP \# 9\CPb$, where the $I_n$ singular fiber, for $n\geq 2$, is a plumbing of $n$ complex 2-spheres of self intersections -2 arranged in a cycle and was given by Kodaira in his famous work \cite{Kod}. Indeed, in \cite{SSS}, Section 8.4 the authors constructed {\em a single} $I_n$ fiber in $E(1)$ from a pencil. However they noted the following fact: ``To understand the other fibers in such a(n elliptic) fibration (over $S^2$) is considerably harder, and when studying more singular fibers, we rather use the monodromy theoretic approach." In our work, without using any monodromies, in a completely geometric way we construct three configurations of $I_n$ fibers with sections in $E(1)$, where each of the configurations has more than one $I_n$ singularity. In each construction we start with a different pencil of cubic curves and obtain a different configuration of $I_n$ fibers with sections. Next, by using these three configurations, in Section \ref{plumbings}, we build four types of star plumbings. For each plumbing, we give three different ways of embedding it inside (some blow-ups of) an elliptic surface, symplectically. Then, in sections \ref{onNoe}, \ref{4th}, \ref{above}, via star surgeries we construct simply connected, minimal, symplectic and exotic 4-manifolds each of which has one Seiberg-Witten basic class up to sign, and lying on the Noether line; between the Noether and half Noether lines; and above the Noether line, respectively.

\subsection{Conventions and Notations}

\label{Conventions}
It is well-known that blow-ups and blow-downs can be done symplectically thanks to McDuff's result \cite{Mc}, and on this paper all blow-ups are performed in the symplectic category. (For an excellent exposition of these operations in the complex and symplectic categories, the reader may see \cite{McDS}). 

Let us end this section by recalling the fiber sum operation. First, an elliptic surface is a complex surface which admits a genus one fibration over a complex curve with finitely many singular fibers. 
We take two elliptic surfaces $S_1,\;S_2$, from each we take out regular neighborhoods of the generic fibers $T^2 \times D^2$. Then we glue the remaining pieces $S_i \setminus (T^2 \times D^2)$ along their boundaries by a fiber preserving, orientation reversing diffeomorphism. This operation is called the fiber sum and the resulting manifold $S_1 \#_f S_2$ also admits an elliptic fibration. In the rest of the paper, $E(n)$ denotes the elliptic surface which is the $n$-fold fiber sum of copies of $E(1) := \CP$, where $E(1)$ is equipped with an elliptic fibration. In particular $E(n) = E(n-1) \#_f E(1)$, with $e(E(n))= 12n$, $\sigma(E(n))=-8n$ and $\pi_1(E(n))= 1$ (Example 5.2 in \cite{G}, Chapter 3 in \cite{GS}). Moreover $E(n)$ could be described as an n-cyclic branched cover of $E(1)$ and hence it admits a complex structure (\cite{GS}, Remark 3.1.8 and 7.3.11). In addition, since $E(n)$ is K\"{a}hler, it is symplectic. In this paper, we consider $E(n)$ as a symplectic 4-manifold.

\subsection*{Acknowledgements} 
I would like to thank Anar Akhmedov for his comments on an earlier draft of this paper and for many helpful discussions. I thank Tian-Jun Li for his comments and pointing out a typo. I am grateful  to \c{C}a\u{g}r{\i} Karakurt and Laura Starkston for many correspondences and their sparing time on my questions. I would like to thank the referee for their constructive and positive remarks which improved this manuscript in great amount. I also acknowledge the financial support and hospitality of the Max Planck Institute for Mathematics, Bonn where most of this work was done during my stay as a postdoctoral fellow.

\section{Background on Seiberg-Witten invariants}
\label{SW}

In this section, let us give some background information on the Seiberg-Witten (SW) invariants by following \cite{Szabo, GS}. Let $X$ be a smooth closed oriented 4-manifold with $b_2^+ (X)>1$. The Seiberg-Witten invariant of $X$ is an integer valued function defined on the set of $\text{spin}^c$ structures over $X$. If $H_1(X,\Z)$ has no 2-torsion we use the one-to-one correspondence between the set of $\text{spin}^c$ structures over $X$ and set of characteristic elements in $H^2(X,\Z)$. After fixing a homology orientation, we have
\begin{equation*}
SW_X :\{K \in H^2(X,\Z) | K \equiv w_2(TX) \text{(mod 2)} \} \rightarrow \Z.
\end{equation*}
$K$ is called a {\em basic class} of $X$ if $SW_X(K) \neq 0$, and ${Bas}_X$ denotes the set of basic classes of $X$.  

Let $g$ be a Riemannian metric on $X$ and $h$ be an arbitrary closed real-valued self-dual 2-form on $X$. Then, the perturbed SW moduli space $\mathcal{M}_X (K, g, h)$ is defined as the solution space of the SW equations
\begin{equation*}
F_A^+ =q(\phi)+ih, \;\; D_A\phi=0
\end{equation*}
divided by the gauge-group, where $A$ is an $S^1$ connection on the line bundle $L$ with $c_1(L) = K$, $F_A^+$ is the self-dual part of the curvature of $A$, $q$ is a certain quadratic map, $\phi$ is a section of the positive spin bundle corresponding to the $\text{spin}^c$ structure determined by $K$, and $D_A$ is the Dirac operator coupled with $A$.

If $b_2^+ (X) \geq 1$ and $h$ is generic, then the moduli space $\mathcal{M}_X (K, g, h)$ is a closed manifold with formal dimension 
\begin{equation}
d = (K^2 - 3 \sigma(X) - 2e(X))/4.
\end{equation}
Here $d < 0$ implies that $\mathcal{M}_X (K, g, h)$ is empty, in this case $SW_X(K) = 0$ by definition. In the $d \geq 0$ case we have
\begin{equation*}
SW_X (K, g, h) := <[\mathcal{M}_X (K, g, h)], \mu^{d/2} >
\end{equation*}
where $\mu \in H^2(\mathcal{M}_X (K, g, h)), \Z)$ is the Euler class of the base fibration.

Recall that a simply connected smooth 4-manifold $X$ is said to be {\em of simple type} if each basic class $K$ satisfies the equation $K^2 = c_1^2(X) = 3 \sigma(X)+2 \chi(X)$. Now let us give a generalized blow-up formula.
\begin{theorem} \cite{GS, FS0}
Assume that a simply connected, smooth 4-manifold $X'$ decomposes as $X'=X \# N$, where $X$ is of simple type. If $b_2^+(N) =0$ from where $H^2(N,\Z)$ has an orthogonal basis $\{E_i \in H^2(N,\Z)\;|\; i=1,2, \cdots, b_2(N)\}$ with $E_i^2 =-1$, then $Bas_{X'} = \{K_i \pm E_1 \pm \cdots \pm E_{b_2(N)} \;|\; K_i \in Bas_X\}$.
\end{theorem}

The basic classes of the elliptic surface $E(n)$, $n\geq2$ are given as follows.

\begin{prop} \cite{FS2}, Corollary 3.1.15 in \cite{GS} 
For $n\geq 2$, $Bas_{E(n)}= \{PD(k \cdot f) \in H^2(E(n),\Z) \; | \; k \equiv n \;\text{(mod 2)}, |k| \leq n-2\}$
\end{prop}
 where $f$ is the homology class of the fiber of $E(n)$ and $PD$ means taking the Poincar\'{e} dual of the homology class.

Let us also recall:

\begin{theorem} \cite{Mich}
Suppose $Y$ is a rational homology sphere which is a monopole $L$-space. Let $P$ and $B$ be negative definite 4-manifolds with $b_1(P) = b_1(B) =0$ and $\partial P = \partial B =Y$. Let $X= Z \cup_Y P$ and $X' = Z \cup_Y B$ for some 4-manifold $Z$. If $s \in Spin^c(X), \; s' \in Spin^c(X')$, $d_X(s), d_{X'}(s') \geq 0$ and $s|_Z = s'|_Z$ then $SW_X(s) = SW_{X'}(s')$.
\label{Mich}
\end{theorem}

\section{Recapping $(\mathcal{Q,R}), (\mathcal{K,L}), (\mathcal{S}_2,\mathcal{T}_2), (\mathcal{U,V})$-star surgeries}
\label{recap}
In this section we will review $(\mathcal{Q,R}), (\mathcal{K,L}), (\mathcal{S}_2,\mathcal{T}_2), (\mathcal{U,V})$-star surgeries from \cite{KS} briefly. Let us begin with the $(\mathcal{Q,R})$ surgery. $\mathcal{Q}$ is the configuration of symplectic spheres which intersect according to a star shaped graph with 4 arms. The central vertex $u_0$ is a -5 sphere, and the arms respectively contain one -3 sphere $u_1$; one -2 sphere $u_2$; -2 and -3 spheres $u_3$ and $u_4$; and lastly two -2 spheres $u_5$ and $u_6$ (see Figure~\ref{Q} below (also Figure 4 in \cite{KS})). 

\begin{figure}[ht]
\scalebox{0.80}{\includegraphics{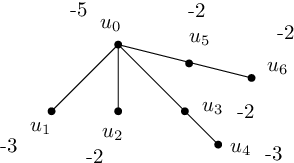}}
\caption{The configuration $\mathcal{Q}$}
\label{Q}
\end{figure}

The intersection form $[\mathcal{Q}]$ for $H_2(\mathcal{Q},\Z)$ is given by a $7 \times 7$ matrix 
\[
\begin{bmatrix}
    -5 &1&1&1&0&1&0\\       
    1&-3&0&0&0&0&0\\
    1&0&-2&0&0&0&0\\
    1&0&0&-2&1&0&0\\
    0&0&0&1&-3&0&0\\
    1&0&0&0&0&-2&1\\
    0&0&0&0&0&1&-2\\
\end{bmatrix}
\]

and its inverse $[\mathcal{Q}]^{-1}$ is

\[-1/261
\begin{bmatrix}
    90 &30&45&54&18&60&30\\       
    30&97&15&18&6&20&10\\
    45&15&153&27&9&30&15\\
    54&18&27&189&63&36&18\\
    18&6&9&63&108&12&6\\
    60&20&30&36&12&214&107\\
    30&10&15&18&6&107&184\\
\end{bmatrix}
\]
The signature $\sigma (\mathcal{Q})$ of $\mathcal{Q}$ is -7. On the other hand, $\mathcal{R}$ is a particular simply connected, symplectic 4-manifold with Euler characteristic 3, signature $-2$, and with convex boundary. In addition the intersection form for $H_2(\mathcal{R},\Z)$ is given by the $2\times 2$ negative definite matrix (Lemma 3.8 in \cite{KS}):
\[\begin{bmatrix}
-10&-23\\
-23&-79\\
\end{bmatrix}\]

We refer the reader to \cite{KS} for the precise Kirby calculus and Lefschetz fibration description of $\mathcal{R}$. Now, let $ \xi_{can}$ denote the canonical contact structure on the boundary $\partial \mathcal{Q}$ of $\mathcal{Q}$. The boundary of $ \mathcal{R}$ with the induced contact structure is contactomorphic to $(\partial \mathcal{Q}, \xi_{can})$ (Proposition 2.6 in \cite{KS}), and $(\mathcal{Q,R})$ surgery is defined as follows.

\begin{definition}
Replacing the neighborhood of $\mathcal{Q}$ in a symplectic 4-manifold by the filling $\mathcal{R}$ is called the $(\mathcal{Q,R})$ surgery.
\label{QR}
\end{definition}

Next, let us recapitulate the $(\mathcal{K,L})$-star surgery from \cite{KS}. $\mathcal{K}$ is the configuration of symplectic spheres which intersect according to a star shaped graph with 4 arms. Each arm contains one -2 sphere $u_i, i=1,\cdots 4$ and the central vertex $u_0$ is a -6 sphere (see Figure~\ref{K}, also figure 6 in \cite{KS}). 

\begin{figure}[ht]
\scalebox{0.80}{\includegraphics{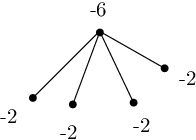}}
\caption{The configuration $\mathcal{K}$}
\label{K}
\end{figure}

The intersection form $[\mathcal{K}]$ for $H_2(\mathcal{K},\Z)$ is given by a $5 \times 5$ matrix 
\[
\begin{bmatrix}
    -6 &1&1&1&1\\       
    1&-2&0&0&0\\
    1&0&-2&0&0\\
    1&0&0&-2&0\\
    1&0&0&0&-2\\
\end{bmatrix}
\]
and its inverse is
\[-1/16
\begin{bmatrix}
     4 &2&2&2&2\\
     2&9&1&1&1\\
     2&1&9&1&1\\
     2&1&1&9&1\\
     2&1&1&1&9    
\end{bmatrix}
\]
The signature $\sigma (\mathcal{K})$ of $\mathcal{K}$ is -5. On the other hand, $\mathcal{L}$ is a particular symplectic 4-manifold with Euler characteristic 2, $c_1(\mathcal{L}) =0$, $\pi_1(\mathcal{L}) = \Z /4$, $H_2(\mathcal{L}) = \Z$ and intersection form is the matrix $[-4]$, hence $\sigma(\mathcal{L}) = -1$. (See \cite{KS} for the precise Kirby calculus and Lefschetz fibration description of $\mathcal{L}$). It is shown that the plumbing $\mathcal{K}$ can be replaced by the symplectic filling $\mathcal{L}$ and we have

\begin{definition}
Replacing the neighborhood of $\mathcal{K}$ in a symplectic 4-manifold by the filling $\mathcal{L}$ is called the $(\mathcal{K,L})$ surgery.
\label{KL}
\end{definition}
In \cite{Star2} it was shown that the $(\mathcal{K,L})$-surgery is not equivalent to any sequences of generalized symplectic rational blow-downs.

Now we will recap the $(\mathcal{S}_2,\mathcal{T}_2)$ surgery of \cite{KS}. Here $\mathcal{S}_2$ is the configuration of symplectic spheres which intersect according to a star shaped graph with four arms. Each arm contains one -2 sphere $u_i, i=1,\cdots 4$ and the central vertex $u_0$ is a -5 sphere (see Figure \ref{S2}, also figure 2 in \cite{KS}).

\begin{figure}[ht]
\scalebox{0.80}{\includegraphics{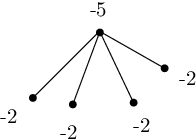}}
\caption{The configuration $\mathcal{S}_2$}
\label{S2}
\end{figure}

The intersection form $[\mathcal{S}_2]$ for $H_2(\mathcal{S}_2,\Z)$ is given by a $5 \times 5$ matrix 
\[
\begin{bmatrix}
    -5 &1&1&1&1\\       
    1&-2&0&0&0\\
    1&0&-2&0&0\\
    1&0&0&-2&0\\
    1&0&0&0&-2\\
\end{bmatrix}
\]
and its inverse is
\[-1/12
\begin{bmatrix}
     4 &2&2&2&2\\
     2&7&1&1&1\\
     2&1&7&1&1\\
     2&1&1&7&1\\
     2&1&1&1&7    
\end{bmatrix}
\]
The signature $\sigma (\mathcal{S}_2)$ of $\mathcal{S}_2$ is -5. On the other hand, $\mathcal{T}_2$ is a particular symplectic 4-manifold with Euler characteristic 3, $\pi_1(\mathcal{T}_2) = \Z /2$, $\sigma(\mathcal{T}_2) = -2$. The intersection form of $\mathcal{T}_2$ is (\cite{KS}, Proposition 3.2)
\[
\begin{bmatrix}
-4 & 0\\
0 & -3
\end{bmatrix}
\]

We again refer the reader to \cite{KS} for the precise Kirby calculus and Lefschetz fibration description of $\mathcal{T}_2$. It is shown that $\mathcal{S}_2$ can be replaced by the symplectic filling $\mathcal{T}_2$ (\cite{KS}). Hence,

\begin{definition}
Replacing the neighborhood of $\mathcal{S}_2$ in a symplectic 4-manifold by the filling $\mathcal{T}_2$ is called the $(\mathcal{S}_2,\mathcal{T}_2)$ surgery.
\label{ST}
\end{definition}

Now let us revisit the $(\mathcal{U,V})$ surgery from \cite{KS} where $\mathcal{U}$ is the configuration of symplectic spheres as in Figure~\ref{U} and $e(\mathcal{U}) = 10$, $\sigma(\mathcal{U}) = -9$. From the intersection form of $\mathcal{U}$, we find its inverse as in Figure \ref{UUinv}. 

\begin{figure}[ht]
\scalebox{0.80}{\includegraphics{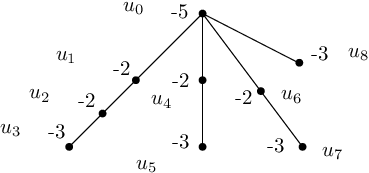}}
\caption{The configuration $\mathcal{U}$}
\label{U}
\end{figure}
\begin{figure}[ht]
\scalebox{0.45}{\includegraphics{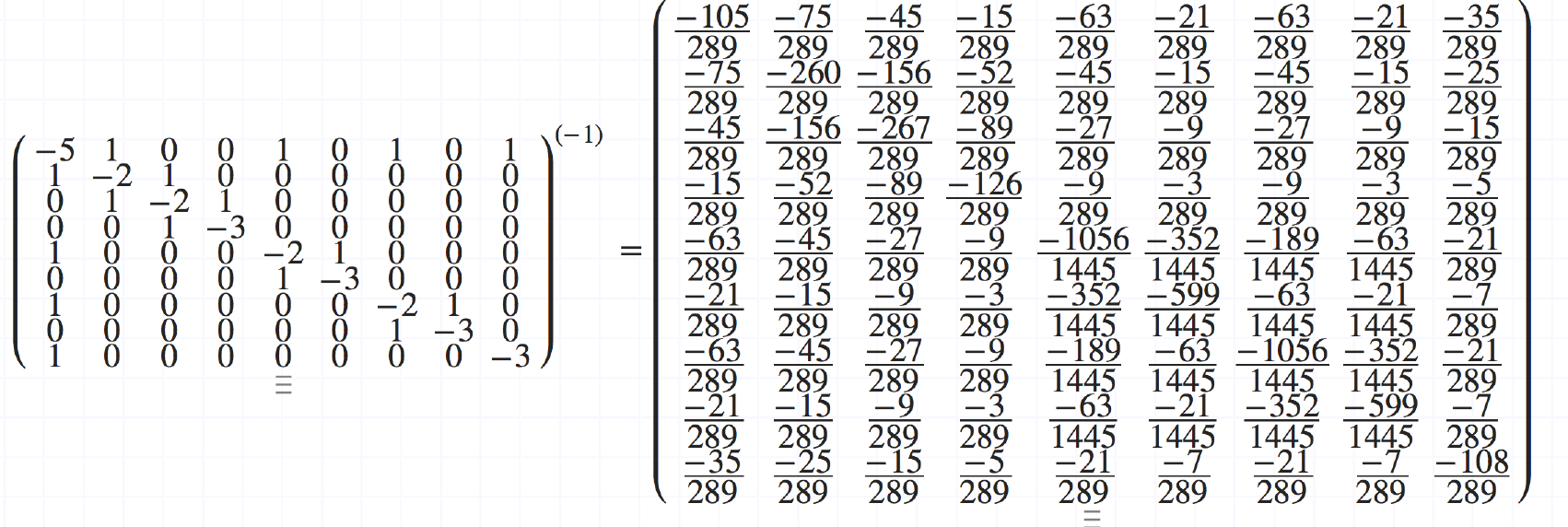}}
\caption{Intersection form $[\mathcal{U}]^{-1}$}
\label{UUinv}
\end{figure}

On the other hand, the symplectic filling $\mathcal{V}$ is a particular symplectic 4-manifold with $e(\mathcal{V})=3$, $\sigma(\mathcal{V}) = -2$. The intersection form of $\mathcal{V}$ is
\[
\begin{bmatrix}
-30 & 5\\
5 & -49
\end{bmatrix}
\]
The precise Kirby calculus and Lefschetz fibration description of $\mathcal{V}$ is given in \cite{KS}.

\begin{definition}
The $(\mathcal{U,V})$ surgery is symplectically replacing the neighborhood of $\mathcal{U}$ in a symplectic 4-manifold by the filling $\mathcal{V}$. 
\label{UV}
\end{definition}

Lastly, let us give the following lemmas which are in fact valid for, not only these above mentioned four types-, but all types of star surgeries of \cite{KS}. We will use them to show that our manifolds constructed in the later sections are exotic and minimal respectively:

\begin{lemma}
Let $M$ be a manifold obtained from a K\"{a}hler manifold from a star surgery operation, where $M$ is simply connected, $b_2^+(M) > 1$, and the intersection form of $M$ is odd and indefinite. Then $M$ is an exotic copy of $n\CP \# m\CPb$, where $n = b_2^+(M)$ and $m = b_2^-(M)$.
\label{exo}
\end{lemma}
\begin{proof}
By Freedman’s classification theorem of simply connected, closed topological 4-manifolds \cite{Freed} $M$ is homeomorphic to $n\CP \# m\CPb$. Star surgeries are symplectic operations \cite{KS} so the resulting manifold $M$ is symplectic. By Taubes' theorem \cite{Tau} on symplectic manifolds, $M$ has nonvanishing Seiberg-Witten invariant. However, by the vanishing theorem for connected sums of manifolds with $b_2^+ > 1$ \cite{GS}, the Seiberg-Witten invariants of $n\CP \# m\CPb$ are all zero. Thus we conclude that $M$ is an exotic copy of $n\CP \# m\CPb$.
\end{proof} 

We need the following lemma to prove that the manifolds we construct in sections 6.1 and 6.2 are minimal. (We would like to thank \c{C}. Karakurt for communicating the proof to us.) 
\begin{lemma}
Let $\mathcal{A}$ be a negative definite plumbing, $[\mathcal{A}]$ be its intersection matrix, and $M,N$ be two characteristic elements in $H^2(\mathcal{A}, \Z)$. Then $M|_{\partial \mathcal{A}} = N|_{\partial \mathcal{A}}$ if and only if the entries of the vector $\frac{1}{2}[\mathcal{A}]^{-1}(M-N)$ are integers.
\label{IMP}
\end{lemma}

\begin{proof}
Let us denote the set of characteristic elements in $H^2(\mathcal{A}, \Z)$ by $Char$. Then the set of $spin^c$ structures on $\mathcal{A}$ is identified with $Char$. If $K \in Char$ and $e \in H_2(\mathcal{A},\Z)$, then $K+2PDe \in Char$, and $K$ and $K+2PDe$ restrict to the same class on $H^2(\partial \mathcal{A},\Z)$. Hence, the restriction of the map: $ H^2(\mathcal{A},\Z) \rightarrow H^2(\partial \mathcal{A},\Z)$ to the $spin^c$ structures is identified with the natural map: $Char \rightarrow Char/2PD(H_2(\mathcal{A}),\Z)$. This gives the isomorphism: $H^2(\partial \mathcal{A},\Z) \simeq Char/2PD(H_2(\mathcal{A}),\Z)$. Hence to check whether $M|_{\partial\mathcal{A}}=N|_{\partial\mathcal{A}}$, we need to check whether $M$ and $N$ are in the same $2PD(H_2(\mathcal{A},\Z))$ orbit. That is to say, we check whether $M -N$ is in the image of $2[\mathcal{A}]$, i.e., whether the entries of the vector $\frac{1}{2}[\mathcal{A}]^{-1}(M-N)$ are integers.
\end{proof} 


\section{Constructions of configurations of $I_n$ singularities in the rational elliptic surfaces}
\label{2nd}

Let us first recall the following facts from \cite{MirPers}. A (Jacobian) rational elliptic surface is the complex projective plane blown up at nine points, $\CP \# 9\CPb$, which admits an elliptic fibration over $\mathbb{CP}^1$ (with a section).  A cubic pencil in $\CP$ is a one-dimensional linear system of cubics, 
which has nine basepoints, counted with multiplicity, by Bezout's theorem.
By blowing up the basepoints of a cubic pencil we obtain a rational elliptic surface. The exceptional divisors of square -1 correspond exactly to the sections. Moreover,

\begin{prop} (\cite{MirPers})
Every Jacobian rational elliptic surface is the blow up of the basepoints of a cubic pencil.
\end{prop}

When we drop the assumption of being Jacobian, a rational elliptic surface is still a blow up of $\CP$ at nine points, though blow-ups are not necessarily at the basepoints of a cubic pencil \cite{MirPers}.  In \cite{Perss} the complete list of singular fibers in the global elliptic fibrations on $\CP \# 9\CPb$ is given, where Persson notes that the configurations other than type $I_1$ and $II$ are obtained by blowing up cubic pencils, where $I_1$ and $II$ denote the fishtail and cusp fibers, respectively. In \cite{Naru}, Naruki explicitly constructs pencils of cubic curves, from which he \emph{states} that $I_n$ fibers, with $n\geq 2$, are obtained in the rational elliptic surfaces. In addition, in \cite{Kuru}, Naruki's work is generalized; more cubic pencils are shown to exist. They also list some $I_n$ configurations with $n\geq 2$, but it is not shown how to obtain these configurations from the given pencils.

In this section, by starting with the pencils in \cite{Naru} and  \cite{Kuru}, we will explicitly construct each of the following configurations in $\CP \# 9\CPb$ that will be used in the later sections of the paper:
\begin{eqnarray}
\label{config}
(I_6,I_3,I_2) \nonumber\\
(I_5,I_4) \nonumber\\
(I_5, I_5) \nonumber
\end{eqnarray}
where the notation $(I_6,I_3,I_2)$ means there is one singular fiber of type $I_6$, one of type $I_3$, and one of type $I_2$. We will find the homology classes of the sphere components of each fiber, verify that their self intersections are -2 and precisely find the -1 sections. Finding homology classes also enables us to do computations (for instance in finding the symplectic Kodaira dimension of the resulting manifolds). In this way, we make these configuration more accessible to work with. Let us note that the $(I_5, I_5)$ configuration is also constructed explicitly in [\cite{KS}, Lemma 4.3]. We also note that we are not listing all of the singular fibers in these cases, we do not consider the global fibrations. We will construct configurations without $I_1$ and $II$ fibers. In fact, let us consider the $(I_6,I_3,I_2)$ configuration. 
The Euler characteristics of the $I_k$ fibers are $e(I_k) = k, \;\;k \geq 2$ and from the simple Euler characteristic computation we see that inside $\CP \# 9\CPb$ there has to be an additional fiber $I_1$. In \cite{Naru}, Naruki shows the existence of $I_1$ fibers via Cremona transformations which correspond to holomorphic automorphisms of the corresponding variety. In our work, we build all singular fibers from cubic pencils and we do not consider additional fishtails and cusps.
\subsection{Construction of the $(I_6,I_3,I_2)$ configuration in $E(1)$}
Let us construct the $(I_6,I_3,I_2)$ configuration in $E(1)$. We will construct this configuration from the pencil $P$ given in \cite{Naru}, Section 2.9. Let us first present $P$ here, from \cite{Naru}, p.332. In $\CP$, Naruki takes a nodal cubic $C$ and a conic $Q$ which intersect only at the node $p$ of $C$ with multiplicity 6. $Q$ touches one of the two branches of $C$ at $p$ with multiplicity 5. Next, he takes an inflection point $q$ of $C$ and the corresponding inflection line $L$. In \cite{Naru}, $L, Q, C$ are given as follows:
\begin{eqnarray*}
L&:& z=0\\
Q&: &x^2+xy+xz+y^2 = 0 \\
C&: &-x^2z+xyz+y^3 = 0
\end{eqnarray*}
He considers the pencil $P$ generated by the two cubic curves $L \cup Q$ and $C$ with base points $p=(0, 0, 1)$ and $q=(1, 0, 0)$. Then, he gives a member $C_1 $ of $P$, which has a node at $q$ and passing through $p$, as follows (\cite{Naru}, p.332):
\begin{equation*}
C_1 : z(x^2+xy+xz+y^2)+(-x^2z+xyz+y^3).
\end{equation*}
The cubic $C_1$ intersects both $L$ and $C$ at the point $q$ with multiplicities are both 3. Let us denote the intersection multiplicities as $(L,C_1)= (C,C_1)=3$. Here one of the branches of the node of $C_1$ at $q$ is tangent to order 2 to $C$ and $L$ (simple tangency), and the other branch intersects both $C$ and $L$ once. We also have $(C,L)=3$ at $q$ and note that $C$ does not have a node at $q$. On the other hand, at the point $p$ of the pencil $P$, it is given that $(C,Q)=(C,C_1)=6$, where one of the branches of the node of $C$ at $p$ is tangent to order 5 to $Q$ and $C_1$ (\cite{Naru}, p.332). Also, $(Q,C_1)=6$ at the point $p$. Lastly, as it is seen from the above equations, $Q$ and $L$ intersect at two distinct points which are different than $p$ and $q$. (cf. \cite{Naru}, Section 2.9, see also the discussion on p.323-324). 

We sketch Naruki's pencil $P$ as in Figure \ref{Pencil} below where we denote $C$ in black, $C_1$ in blue and $Q, L$ in green, and the total intersection multiplicities at the points $q$ and $p$ by $\times 3$ and $\times 6$, respectively.


\begin{figure}[ht]
\scalebox{1.3}{\includegraphics{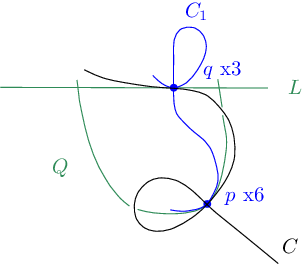}}
\caption{Pencil of cubic curves}
\label{Pencil}
\end{figure}

In \cite{Naru}, it is stated that by blowing up $\CP$ three times over $q$ and six times over $p$, one can obtain $I_3$ and $I_6$ fibers, respectively, in $\CP \# 9 \CPb$. Moreover, it is claimed that the strict transform of $Q\cup L$ gives the $I_2$ fiber, and as a result of blow-ups, two of the $-1$ sections are obtained (\cite{Naru}, p.332). 

Now we will verify these claims by explicitly constructing $I_6, I_3, I_2$ fibers and the two $-1$ sections. We proceed as in \cite{SS1}, Section 5 or \cite{KS}, Section 4. Namely, we start with the pencil $P$, blow-up the base points $p$ and $q$, and keep track of the intersection multiplicities of the intersection points after each blow-up. Moreover, initially the total homology classes of each of the black, blue and green parts, $C$, $C_1$ and $Q\cup L$ (see Figure \ref{Pencil}), of the pencil are $3h$. After each blow up, we compute the total homology classes of the proper transforms of these three parts in different colors. To equate their total homology classes, we include the exceptional spheres coming from the blow-ups into one of these three parts. This determines the new points to be blown-up. We continue until all intersections are resolved and the total homology classes of the three different parts become equal. Let us explain this process in more detail. 

We start with the pencil $P$ which is depicted in Figure \ref{Pencil} and we blow up the points $q$ and $p$. We obtain the exceptional spheres $e_1$ and $e_2$, respectively, and the nodes of the cubics $C_1$ and $C$ are resolved as we show in the first configuration of Figure \ref{632}. After resolving the node of $C_1$ at $q$ we obtain $(\tilde L, \tilde C_1)= (\tilde C, \tilde C_1)=1$, and at the same point we have $(\tilde C, \tilde L)=2$, where we denote the proper transforms of the curves $L,C,C_1$ by $\tilde L, \tilde C, \tilde C_1$. Moreover, after blowing up the double point of $C$ at $p$ we have $(\tilde C, \tilde Q) = (\tilde C, \tilde C_1) = 4$, and $(\tilde Q, \tilde C_1)=5$. Hence we obtain the following homology classes after the blow-ups (as we also show in the first configuration of Figure \ref{632}):
\begin{eqnarray*}
\tilde C_1 &=& 3h-2e_1-e_2\\
\tilde L &=& h-e_1, \; \tilde Q = 2h-e_2\\
\tilde C &=& 3h - e_1-2e_2
\end{eqnarray*}
Initially in the pencil $P$ the homology classes of $C_1$, $L\cup Q$ and $C$ are all $3h$. After blowing-up at $q$ and $p$ we need to equate the total homology classes of these three parts, $\tilde C_1, \tilde L \cup \tilde Q$, and $\tilde C$ of the configuration that are shown in blue, green and black, respectively. From the above equations, the total homology class of $\tilde L \cup \tilde Q$ is $3h-e_1-e_2$. Therefore, we add $e_1$ to $\tilde C_1$ (to the blue part), and $e_2$ to $\tilde C$ (to the black part). Next we blow up at the two points indicated in black in the first configuration where the intersection multiplicities are also shown. We obtain the second configuration in Figure \ref{632} where $e_3$ and $e_4$ are the exceptional spheres. We keep track of the intersection multiplicities and homology classes as we show at each step in the figure. By abuse of notation, after each blow-up we show the proper transform of each curve by the same notation but we write down the homology classes of these curves. 

In the second configuration, to equate the total homology classes of the blue, green and black parts, we add $e_3$ to the blue part and $e_4$ to the black part. Each of the three total classes of the blue, green and black parts becomes $3h-e_1-e_2-e_3-e_4$. Next we blow-up the indicated black points in the second configuration of the figure, and we get the third configuration with $e_5, e_6$ the exceptional spheres. (The intersection multiplicities are also given in the figure). To equate the classes, we only add $e_6$ to the black part. At the $-1$ sphere $e_5$, the intersection of blue, green and black parts are separated. Each of the total classes of the three parts become $3h-e_1-e_2-e_3-e_4-e_5-e_6$ in the third configuration. Then we blow-up at the indicated black point of the third configuration to separate the blue, green and black curves. We get $e_7$ as the exceptional sphere as shown in the fourth configuration. Next we add $e_7$ to the black part. We compute the total classes of each of the three parts, each of them is equal to $3h-e_1-e_2-e_3-e_4-e_5-e_6-e_7$. We blow-up the black point in the fourth configuration, which gives the fifth one where $e_8$ is the exceptional sphere. We add $e_8$ to the black part to equate the total homology classes. Now each of the black, blue and green parts has $3h-e_1-e_2-e_3-e_4-e_5-e_6-e_7-e_8$ as their total class. We blow-up the intersection point indicated in black where we have $(\tilde Q, \tilde C_1)=1$ and it is the only remaining intersection. Hence we obtain the last configuration. 

In the last step of Figure \ref{632}, we compute the total homology classes of the blue, black and green parts again, which are all equal to $3h-e_1 - e_2 \cdots - e_9$ (see the last step):
\begin{eqnarray*}
3h-e_1 - e_2 \cdots - e_9 &=& \tilde C_1 +(e_1-e_3) + (e_3-e_5)\\ 
&=& \tilde C + (e_2-e_4)+(e_4-e_6)+(e_6-e_7)+(e_7-e_8)+(e_8-e_9)\\
 &=& \tilde L + \tilde Q
\end{eqnarray*}
Hence the blue, black and green parts are completely separated, there are no more points to be blown-up, and the process stops here. We have $\tilde C_1=3h-2e_1 -e_2 -e_3 -e_4 -e_6 -e_7 -e_8 -e_9$, $(e_1-e_3)$ and $(e_3-e_5)$ that are all $-2$ spheres and they intersect each other as shown in the figure. So they give the $I_3$ fiber. In addition, we obtain the following six $-2$ spheres $\tilde C = 3h-e_1 -2e_2 -e_3 -e_4 -e_5 -e_6 -e_7 -e_8$, $(e_2-e_4),(e_4-e_6),(e_6-e_7),(e_7-e_8),(e_8-e_9)$. They give the $I_6$ fiber. Lastly, the strict transforms of the conic $Q$ and the line $L$ are $\tilde Q =2h-e_2-e_4-e_6-e_7-e_8-e_9$ and $\tilde L= h - e_1- e_3- e_5$. So, they are both $-2$ spheres and they intersect at two distinct points, thus they give the $I_2$ fiber (see the last configuration in Figure \ref{632}). Hence we obtain the $(I_6,I_3,I_2)$ configuration in $\CP \# 9 \CPb = E(1)$. We have also obtained the two $-1$ sections, as claimed by Naruki, which are $e_5$ and $e_9$ as shown in the last step of Figure \ref{632}.


\begin{figure}
\scalebox{0.60}{\includegraphics{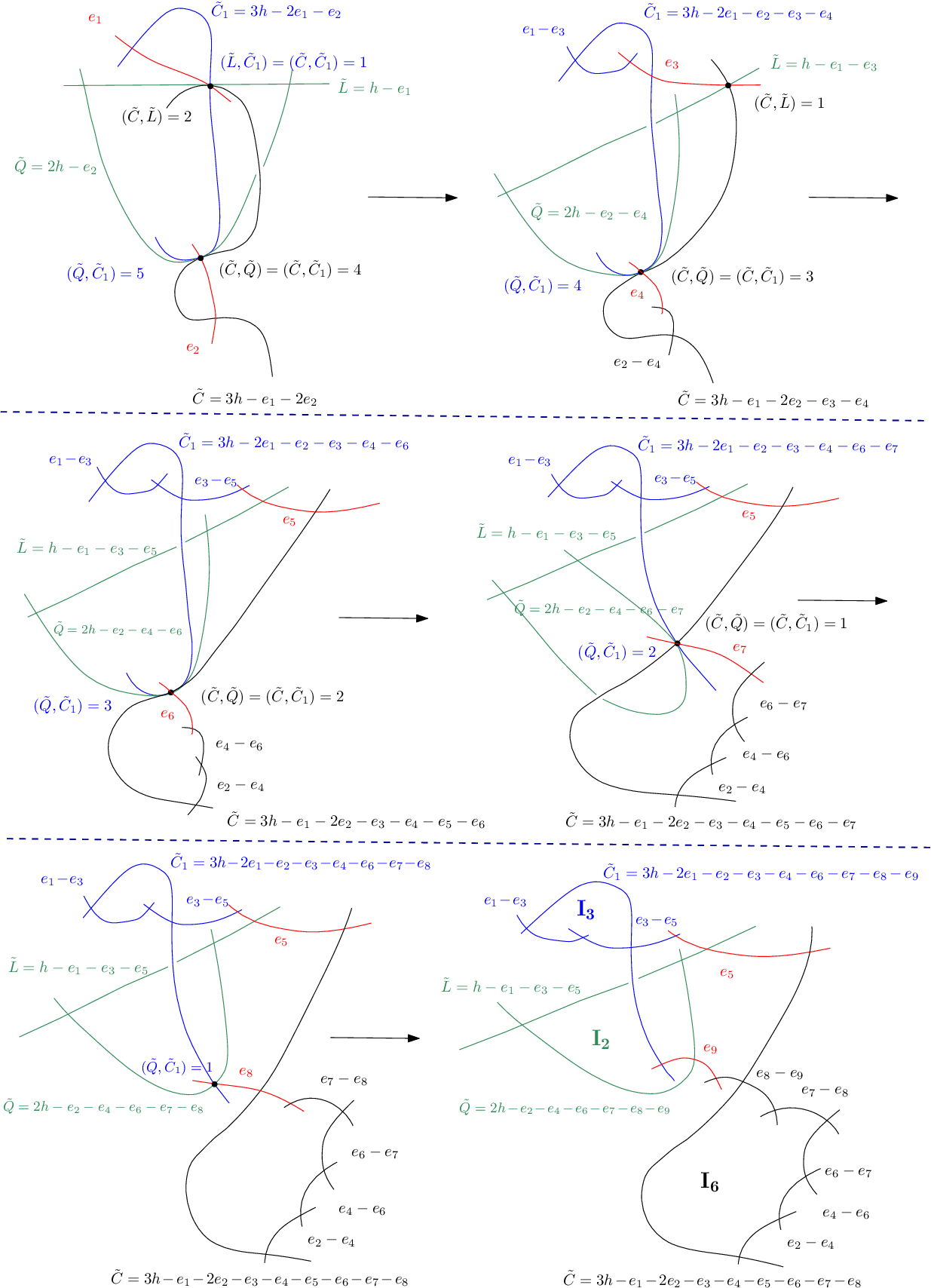}}
\caption{Construction of the $(I_6,I_3,I_2)$ configuration in $E(1)$}
\label{632}
\end{figure}


\subsection{Construction of the $(I_5,I_4)$ configuration in $E(1)$}
In this section we construct the $(I_5,I_4)$ configuration in $E(1)$ by starting with the pencil given in \cite{Kuru}, p.22, No.55, ii). The equations of this pencil are given as follows (\cite{Kuru}, p.16, ii)):
\begin{eqnarray*}
L&:& x+y=0\\
Q&: &z^2+xy= 0 \\
S &:& x=0\\
A&=& (0,1,0), B= (1,-1,1), C=(0,0,1), E=(1,-1,-1)\\
P&=& (-a^2:1:a), a \neq 0
\end{eqnarray*}
Here $Q$ is the conic in $\CP$, and any conic in $\CP$ is isomorphic to $\mathbb{CP}^1$ \cite{Harts}. The line $L$ and the conic $Q$ intersect at the points $B$ and $E$. $S$ is the tangent line to $Q$ at the point $A$. Next, Kurumadani takes the lines passing through $A,B$ and $E,P$. Let us call these lines $K$ and $M$, respectively. The generators of this pencil are $Q \cup L$, and $S \cup K \cup M$. We sketch this pencil as in the first step of Figure \ref{54}, where $Q$ and $L$ are in blue, and $K,M,S$ are in black.
\begin{figure}[ht]
\scalebox{0.60}{\includegraphics{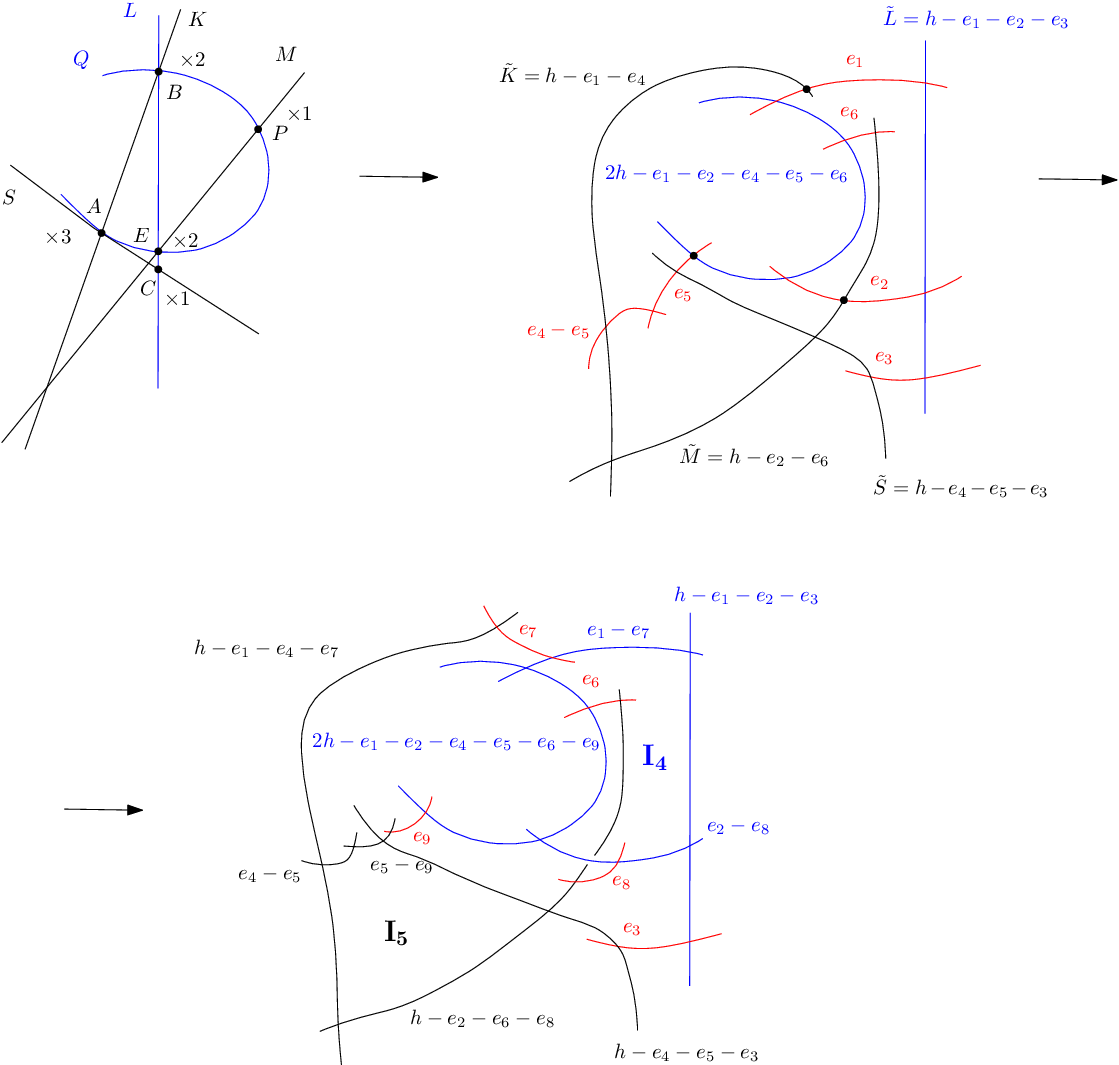}}
\caption{Construction of the $(I_5, I_4)$ configuration in $E(1)$}
\label{54}
\end{figure}
We blow-up the points $B,P,E,C,A$ and we denote the exceptional spheres $e_1, e_6, e_2, e_3, e_4$, respectively. When we blow-up the point $A$, the line $K$ is separated from $Q$ and $S$, and $K$ hits the $-1$ sphere $e_4$. But since $S$ and $Q$ have a tangency at $A$, they are not resolved after one blow-up. They intersect once where $e_4$ also passes through this intersection. So, we do one more blow-up and obtain the exceptional sphere $e_5$, and $e_4$ becomes $e_4-e_5$. Hence $\tilde K$ hits $e_4-e_5$ and $\tilde S$ hits $e_5$. See the second part of Figure \ref{54} where we also give the homology classes of each curve. Now we compute the total homology classes of the blue and black parts in the second step:
\begin{eqnarray*}
\tilde L + \tilde Q&=& (h-e_1-e_2-e_3) + (2h-e_1-e_2-e_4-e_5-e_6)\\ 
&=& 3h -2e_1 -2e_2 -e_3-e_4-e_5-e_6,\\
\tilde K+ \tilde M+ \tilde S&=& (h-e_1-e_4)+(h-e_2-e_6)+(h-e_4-e_5-e_3)\\ 
&=& 3h -e_1 -e_2 -e_3-2e_4-e_5-e_6
\end{eqnarray*}
To equate them, we add $e_1$ and $e_2$ to the blue part, and $e_4-e_5$ and $e_5$ to the black part. Now both of the two total homology classes become equal to $3h -e_1 -e_2 -e_3-e_4-e_5-e_6$. To separate the blue and black parts, we blow up the 3 black points on $e_1,e_2,e_5$ as shown in the second configuration of Figure \ref{54}. Thus we obtain the third configuration. In addition to $e_3$ and $e_6$ from previous step, we get the $-1$ spheres $e_7, e_8,e_9$ as sections. The four $-2$ spheres $(2h-e_1-e_2-e_4-e_5-e_6-e_9), (h-e_1-e_2-e_3), (e_1-e_7),(e_2-e_8)$ give the $I_4$ fiber, and the five $-2$ spheres $(h-e_1-e_4-e_7),  (h-e_2-e_6-e_8), (h-e_4-e_5-e_3), (e_4-e_5), (e_5-e_9)$ give the $I_5$ fiber as shown in the last part of Figure \ref{54}. Hence we acquire the $(I_5, I_4)$ configuration in $E(1)$.

\subsection{Construction of the $(I_5, I_5)$ configuration in $E(1)$}
Lastly, let us give the construction of the $(I_5, I_5)$ configuration in $E(1)$. We start with the pencil given in \cite{Naru}, Section 2.10. Let us present this pencil here. Naruki takes four points in general position in $\CP$:
$p_1=(1, -1,0), p_2 = (1,0,-1), q_1 = (0,1,0), q_2=(0,0,1)$. Then he takes two reducible cubics (each of them is a union of 3 lines): $\Delta_1 = \overline{p_1q_1} \cup \overline{p_2q_2} \cup \overline{p_1p_2}$ and $\Delta_2 = \overline{p_1q_2} \cup \overline{p_2q_1} \cup \overline{q_1q_2}$ as generators of the cubic pencil. He also gives the equations of the generators as follows (\cite{Naru}, Section 2.10, p.334):
\begin{eqnarray*}
\Delta_1 &=&yz(x+y+z)=0\\
\Delta_2 &=&x(x+y)(x+z)=0
\end{eqnarray*}
Moreover, it is given that $p_1, p_2, q_1,q_2$ and $r=(0, 1, -1)$ are the base points of this pencil. We sketch this line arrangement $\Delta_1 \cup \Delta_2$ as in the first part of Figure \ref{2I5}. By blow-ups we obtain two $I_5$ fibers in $E(1)$ as depicted in the second part of Figure  \ref{2I5}. However, we note that $\Delta_1 \cup \Delta_2$ is the same line arrangement as given in \cite{KS}, Figure 14, and in figures 15 and 16 they give the construction steps of two $I_5$ fibers (\cite{KS}, p.1614). Therefore, we skip the details here.

\begin{figure}
\scalebox{0.60}{\includegraphics{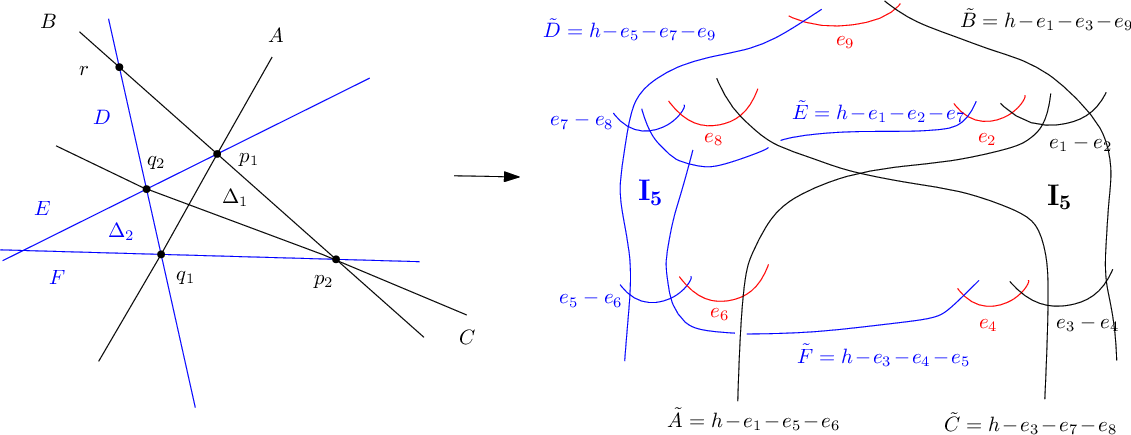}}
\caption{Construction of the $(I_5, I_5)$ configuration in $E(1)$}
\label{2I5}
\end{figure}

\section{Constructions of the plumbings $\mathcal{Q}, \mathcal{K}, \mathcal{S}_2, \mathcal{U}$ from the $(I_6,I_3, I_2), (I_5, I_4)$ and $(I_5, I_5)$ configurations}
\label{plumbings}

In the following sections we will construct exotic 4-manifolds by the $(\mathcal{Q,R}), (\mathcal{K,L})$, $(\mathcal{S}_2,\mathcal{T}_2)$, $(\mathcal{U,V})$-star surgeries in some blow-ups of the manifolds $E(n)$'s. In this section we will construct the plumbings $\mathcal{Q}, \mathcal{K}, \mathcal{S}_2, \mathcal{U}$ from the $(I_6,I_3, I_2), (I_5, I_4)$ and $(I_5, I_5)$ configurations which we built in the previous section in $E(1)$. Now we consider these configurations in the manifolds $E(n)$. (Note that when we take the $n$-fold fiber sum of $E(1)$'s, we sew the $-1$ sections and obtain a $-n$ section of the resulting manifold $E(n)$. In addition, if in each copy of $E(1)$ we take the same type of configuration, after the $n$-fold fiber sum we obtain $n$ copies of that configuration inside $E(n)$). Let us give the following figures \ref{1}, \ref{2}, \ref{3}, where we label the $-2$ spheres of the fibers, the $-n$ section of $E(n)$ and some intersection points. In figures \ref{1}, and \ref{3} we take one copy each of $(I_6,I_3,I_2)$ and $(I_5,I_5)$ configurations, but in Figure \ref{2} we take two of the $n$ copies of the $(I_5,I_4)$ configuration in $E(n)$. By using these configurations we prove the following lemmas.

\begin{figure}[ht]
\scalebox{0.75}{\includegraphics{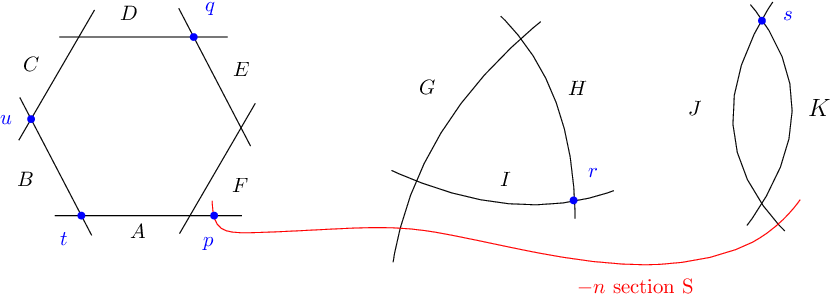}}
\caption{$(I_6,I_3,I_2)$ configuration in $E(n)$}
\label{1}
\end{figure}

\begin{figure}[ht]
\scalebox{0.70}{\includegraphics{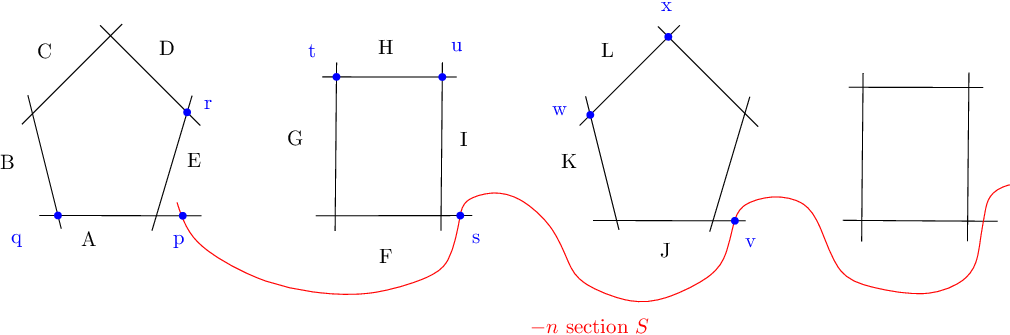}}
\caption{Two copies of $(I_5,I_4)$ configuration in $E(n)$}
\label{2}
\end{figure}

\begin{figure}[ht]
\scalebox{0.80}{\includegraphics{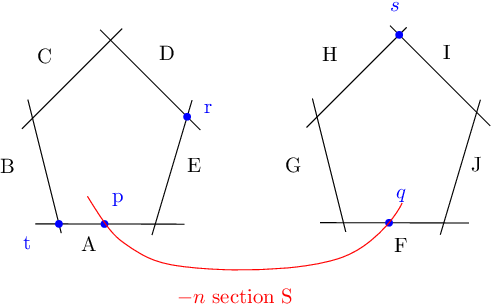}}
\caption{$(I_5,I_5)$ configuration in $E(n)$}
\label{3}
\end{figure}

\begin{lemma}
The plumbing $\mathcal{Q}$ (given in Figure \ref{Q}) symplectically embeds in $E(5)\#\CPb$.
\label{Qin}
\end{lemma}

\begin{proof}
We will prove this lemma in three ways. First we consider the $(I_6,I_3,I_2)$ configuration in $E(5)$ as in Figure \ref{1}, where $S$ is the $-5$ section. We do symplectic resolution at the points $p,t,u$. Thus the symplectic resolution $S+A+B+C$ of the spheres $S,A,B,C$ gives the $-5$ sphere which we take as the central vertex of the plumbing $\mathcal{Q}$. Then we blow-up the point $q$ which is the intersection point of the $-2$ spheres $D$ and $E$. Next, we take the following spheres as for the four arms of the plumbing $\mathcal{Q}$
\begin{eqnarray*}
\tilde D : (-3)\\
J: (-2)\\
F, \tilde E: (-2,-3)\\
G,H: (-2,-2)
\end{eqnarray*}
where we write down the ordered self intersections of the spheres in parentheses. For instance $F, \tilde E: (-2,-3)$ means that the self intersections of $F, \tilde E$ are $-2,-3$, respectively. Hence we obtain $\mathcal{Q}$ symplectically embedded in $E(5)\#\CPb$. Alternatively, let us take two copies of the $(I_5,I_4)$ configuration in $E(5)$ as in Figure \ref{2}, where $S$ is the $-5$ section. We symplectically resolve the intersection points $p,q$ as shown in the figure. The symplectic resolution $S+A+B$ gives the $-5$ sphere which we take as the central vertex of $\mathcal{Q}$. We blow-up the point $r$. Then we take the following spheres, as for the four arms of $\mathcal{Q}$
\begin{eqnarray*}
\tilde E : (-3)\\
J: (-2)\\
C, \tilde D : (-2, -3)\\
F,G: (-2,-2)
\end{eqnarray*}
This gives the plumbing $\mathcal{Q}$ symplectically embedded in $E(5)\#\CPb$. We can also build $\mathcal{Q}$ from the $(I_5,I_5)$ configuration in $E(5)$ shown in Figure \ref{3} with $S$ the $-5$ section. We symplectically resolve the points $p,t,q$. The symplectic resolution $S+A+B+F$ gives the central vertex of $\mathcal{Q}$. We blow-up the point $r$. For the arms we take 
\begin{eqnarray*}
\tilde E : (-3)\\
G: (-2)\\
C, \tilde D : (-2, -3)\\
J,I: (-2,-2)
\end{eqnarray*}
This gives yet another proof of the lemma.
\end{proof}

\begin{lemma}
The plumbing $\mathcal{K}$ (given in Figure \ref{K}) symplectically embeds in $E(6)$, and the plumbing $\mathcal{S}_2$ as in Figure \ref{S2} symplectically embeds in $E(5)$.
\label{KSin}
\end{lemma}

\begin{proof}
We again give three alternative proofs for these symplectic embeddings as in the previous lemma. Let us start with the plumbing $\mathcal{K}$. We take $(I_6,I_3,I_2)$ in $E(6)$ as in Figure \ref{1}, where $S$ is the $-6$ section. We symplectically resolve the intersection of $S$ and $A$, this gives the $-6$ central vertex of $\mathcal{K}$. The $-2$ spheres $B,F, G,J$ are the four arms of $\mathcal{K}$. This gives the first embedding. Alternatively, we take one $(I_5,I_4)$ configuration in $E(6)$. Let us take the first two fibers in Figure \ref{2}. We symplectically resolve the intersection points $p$ and $s$. The symplectic resolution $A+S+F$ gives the $-6$ central vertex, the $-2$ spheres $B,E, G,I$  give the four arms of $\mathcal{K}$ and we get the second embedding. Now we consider the $(I_5,I_5)$ configuration in $E(6)$ (Figure \ref{3}). From the symplectic resolutions at the points $p$ and $q$, we obtain the symplectic $-6$ sphere $A+S+F$ which is the central vertex, and the $-2$ spheres $B,E,G,J$ are the four arms of $\mathcal{K}$. This gives the last embedding of $\mathcal{K}$. 

In the above proofs, if we replace $E(6)$'s by $E(5)$'s (and take $S$ as the $-5$ section), we get the desired symplectic embeddings of the plumbing $\mathcal{S}_2$ in $E(5)$.
\end{proof}

\begin{lemma}
The plumbing $\mathcal{U}$ (given in Figure \ref{U}) symplectically embeds in $E(5)\#3\CPb$.
\label{Uin}
\end{lemma}

\begin{proof}
We take $(I_6,I_3,I_2)$ in $E(5)$ as in Figure \ref{1}, where $S$ is the $-5$ section. We symplectically resolve the intersection of $S$ and $A$ and obtain the symplectic $-5$ sphere $S+A$ as the central vertex of $\mathcal{U}$. Next, we blow-up the points $q,r,s$ as shown in Figure \ref{1}, and for the four arms of $\mathcal{U}$ we take
\begin{eqnarray*}
B,C, \tilde D : (-2,-2,-3)\\
F, \tilde E : (-2, -3)\\
G, \tilde H: (-2,-3)\\
\tilde J: (-3).
\end{eqnarray*}
This proves the lemma. Alternatively, let us consider two copies of the $(I_5,I_4)$ configuration in $E(5)$ as in Figure \ref{2}, where $S$ is the $-5$ section. We symplectically resolve the intersection point $p$ and get the central vertex $A+S$. Then we blow-up at the points $r,t,w$. The four arms of $\mathcal{U}$ are:
\begin{eqnarray*}
B,C, \tilde D : (-2,-2,-3)\\
F, \tilde G : (-2, -3)\\
J, \tilde K : (-2,-3)\\
\tilde E: (-3).
\end{eqnarray*}
Hence we again constructed $\mathcal{U}$ embedded in $E(5)\#3\CPb$ symplectically. 
\begin{remark}
We would like to note that in fact \;$\mathcal{U}$ symplectically embeds in $E(5)\#2\CPb$ by using the $(I_5,I_5)$ configuration in $E(5)$. We see this as follows. For the central vertex of\; $\mathcal{U}$, we take the symplectic $-5$ sphere $A+S+F$ obtained from the symplectic resolutions of the points $p$ and $q$ as shown in Figure \ref{3}. We blow-up only two points $r,s$, and for the four arms of $\mathcal{U}$ we take
\begin{eqnarray*}
B,C, \tilde D : (-2,-2,-3)\\
G, \tilde H : (-2, -3)\\
J, \tilde I: (-2,-3)\\
\tilde E: (-3).
\end{eqnarray*}
Hence $\mathcal{U}$ symplectically embeds in $E(5)\#2\CPb$, thus also in $E(5)\#3\CPb$. This gives the third proof of the lemma. However, from the simple Euler characteristic and signature computation we see that the $(\mathcal{U},\mathcal{V})$ star surgery applied to $E(5)\#2\CPb$ gives us a symplectic manifold above the Noether line, with $\chi_h=5, c_1^2 = 5$. But since in this paper our main interest is to construct manifolds on and below the Noether line, we will mainly consider $\mathcal{U}$ in $E(5)\#3\CPb$, so we state the lemma with $E(5)\#3\CPb$.
\label{Uin2blowup}
\end{remark}
\end{proof}

Let us close this section by a remark and a question.

\begin{remark}
In this section we have given the symplectic embeddings of the four plumbings in (some blow-ups of) $E(5)$ and $E(6)$ to construct exotic manifolds on and below the Noether line. In fact it is possible to obtain the same plumbings inside smaller manifolds as follows. Assume that a star shaped plumbing $\mathcal{A}$ symplectically embeds in $E(m)\#n\CPb$ $m>1, n\geq0$, where the central vertex of $\mathcal{A}$ is obtained from the symplectic resolutions of the $-m$ section with some of the intersecting $-2$ spheres of the singular fibers. Now let us consider $E(k)$ for $k<m$, and assume that $E(k)$ has the same configuration of $I_n$ fibers as in $E(m)$. We blow up the $-k$ section at $(m-k)$ distinct points away from the fibers, which gives a $-m$ sphere. Then, since we have the same configuration and obtained the same central vertex, we obtain the plumbing $\mathcal{A}$ in $E(k)\#(m-k+n)\CPb$ which is a smaller manifold. (See Section \ref{above} for some applications of this remark.)
\label{kmgeneral}
\end{remark}
\begin{question}
In lemmas \ref{Qin}, \ref{KSin}, \ref{Uin}, for each of the $\mathcal{Q}, \mathcal{K}, \mathcal{S}_2, \mathcal{U}$ plumbings, $\mathcal{A}$, we showed that $\mathcal{A}$ symplectically embeds in the corresponding manifold $M$ in three different ways, namely, via the $(I_6,I_3, I_2), (I_5, I_4)$ and $(I_5, I_5)$ configurations. We ask the following. Is there a symplectomorphism of $M$ to itself which takes one of the embeddings of $\mathcal{A}$ to another? It is an interesting problem. 

In our constructions in the following sections, we will take the first embedding of each of the plumbings which comes from the $(I_6, I_3, I_2)$ configuration, unless otherwise stated.
\end{question}

\section{Constructions of simply connected, minimal, symplectic and exotic 4-manifolds on the Noether line}
\label{onNoe}
In this section we construct simply connected, minimal, symplectic 4-manifolds $X$ and $T$ with exotic smooth structures, lying on the Noether line and each with one basic class up to sign. We construct $X$ and $T$ via $(\mathcal{Q,R})$ and $(\mathcal{U,V})$ star surgeries, respectively. 

\subsection{First construction via the $(\mathcal{Q,R})$-star surgery}
Let us begin with the first construction which is by the $(\mathcal{Q,R})$ star surgery. We have shown that $\mathcal{Q}$ symplectically embeds in $E(5) \#\CPb$ (Lemma \ref{Qin}). Let  
\begin{equation*}
X :=( W \setminus \mathcal{Q}) \cup \mathcal{R}
\end{equation*}
where $W := E(5)  \# \CPb$. Then $\sigma(X) = \sigma(W) - \sigma(\mathcal{Q}) + \sigma(\mathcal{R})= -41+7-2 = -36$ and $e(X) = e(W) - e(\mathcal{Q}) + e(\mathcal{R})=61-8+3 = 56$.
Thus, 
\begin{equation}
\chi_h(X)=5 \;\;\;\text{and}\;\;\;  c_1^2(X) = 4= 2\chi_h-6
\end{equation}
which shows that $X$ is on the Noether line. From Van Kampen's theorem, we easily see that $X$ is simply connected as $\mathcal{R}$ is simply connected. 
By Lemma \ref{exo} we conclude that $X$ is an exotic copy of $9\CP \# 45 \CPb$.
\label{X}


Next, let us prove that $X$ is minimal. We take a spin-c structure $s$ on $X$ and we look at its restriction to the filling $R$ and to the boundary $\partial R = \partial Q$ of $R$. We would like to show that the restriction $s|_{\partial R}$ extends over $Q$. 
To this end, we use the long exact sequence of a pair for $(Q,\partial Q)$. We see that the restriction map $H^2(Q) \rightarrow H^2(\partial Q)$ is surjective because $H_1(Q) = H^3(Q,\partial Q) = 0$. Therefore every spin-c structure on $\partial Q = \partial R$ extends to $Q$. In particular, any spin-c structure on $X \setminus R$ extends to $(X \setminus R) \cup Q = W$. 
(see also Theorem \ref{Mich} above). Now let us check which basic classes of $W$ extends to $X$. Seiberg-Witten basic classes of $W= E(5) \# \CPb$ are $\pm f \pm E_1$ and $\pm 3f \pm E_1$ where $f, E_1 \in H^2(W,\Z)$ are the Poincar\'{e} duals of the homology classes of the regular fiber and the exceptional sphere coming from the blow-up, respectively (\cite{GS}). Let $P = f+ E_1$ and  $\gamma_0,..., \gamma_6$ be the basis of $H^2(\mathcal{Q},\Q)$ which is dual to $u_0,...,u_6$ (i.e., we have $\gamma_i(u_j)= \delta_{ij}$). (See Figure \ref{Q} for the spheres $u_i$). Then 
\begin{eqnarray*}
P|_\mathcal{Q} &= (P \cdot u_0)\gamma_0 +  (P \cdot u_1)\gamma_1 +  (P \cdot u_4)\gamma_4
&=  \gamma_0 + \gamma_1+\gamma_4.
\end{eqnarray*}
From inverse of the intersection matrix $[\mathcal{Q}]$ we find that
\begin{equation}
(P|_\mathcal{Q})^2 = -1.54
\end{equation}

Now let us assume that there is a basic class $\widetilde P$ on $X$ such that $\widetilde P |_{X-\mathcal{R}} = P |_{W-\mathcal{Q}}$. Then $(\widetilde P|_{\mathcal{R}})^2 \leq 0$, since $\mathcal{R}$ is negative definite. In fact, the contact 3-manifolds on the boundary of a star surgery plumbing is always planar by a construction in \cite{GayMark}. As a consequence, the fillings are always negative definite. 
Therefore the dimension of the SW moduli space satisfies the following:
\begin{eqnarray*}
d_X(\widetilde P) &=& \frac{\displaystyle{\widetilde P^2 -3\sigma(X)-2\chi(X)}} {\displaystyle{4}}\\
&=&\frac{\displaystyle{P^2 - (P|_{\mathcal{Q}})^2+ (\widetilde P|_{\mathcal{R}})^2 -3\sigma(X)-2\chi(X)}} {\displaystyle{4}}\\
&=& \frac{\displaystyle{-1 +1.54 +  (\widetilde P|_{\mathcal{R}})^2 -4}}{\displaystyle{4}}\\
&=& \frac{\displaystyle{-5 +1.54 +  (\widetilde P|_{\mathcal{R}})^2}}{\displaystyle{4}}\\
&<&0
\end{eqnarray*}
This contradicts our assumption that $\widetilde P$ is a basic class of $X$.

Next, let $L = f-E_1$ and assume that there is a basic class $\widetilde L$ on $X$ such that $\widetilde L |_{X-\mathcal{R}} = L |_{W-\mathcal{Q}}$. Similarly as above,
\begin{eqnarray*}
L|_\mathcal{Q} &=&  \gamma_0 - \gamma_1-\gamma_4,\\
(L|_\mathcal{Q})^2 &=& -0.8
\end{eqnarray*}
which implies
\begin{equation*}
d_X(\widetilde L) = \frac{\displaystyle{-5 + 0.8 +  (\widetilde L|_{\mathcal{R}})^2}}{\displaystyle{4}} <0
\end{equation*}
where $(\widetilde L|_{\mathcal{R}})^2 \leq 0$, as $\mathcal{R}$ is negative definite. Hence, again we reach a contradiction.

Let us look at the class $N:= 3f-E_1$. Under the assumption that there is a basic class $\widetilde N$ on $X$ such that $\widetilde N |_{X-\mathcal{R}} = N |_{W-\mathcal{Q}}$, we have
\begin{eqnarray*}
N|_\mathcal{Q} &=&  3\gamma_0 - \gamma_1-\gamma_4,\\
(N|_\mathcal{Q})^2 &=&9(\gamma_0)^2 + (\gamma_1)^2 + (\gamma_4)^2 -6 \gamma_0\gamma_1-6 \gamma_0\gamma_4+ 2 \gamma_1\gamma_4\\
&=&-1/261 (810+97+108-180-108+12)\\
&=& -2.83
\end{eqnarray*}
Therefore,
\begin{equation*}
d_X(\widetilde N) = \frac{\displaystyle{-5 + 2.83 +  (\widetilde N|_{\mathcal{R}})^2}}{\displaystyle{4}} <0.
\end{equation*}
This contradiction shows that $N$ does not extend to $X$ as a basic class, either. 

However, up to sign, the last class $M := 3f +E_1$ extends to the symplectic manifold $X$ as a basic class. In fact, if it did not extend as a basic class of $X$, this would contradict the fact that $X$ has at least one pair of basic classes by Taubes' theorem (\cite{Tau}). Hence we conclude that only the class $M$ extends to $X$ as a basic class. To prove that $X$ is minimal by the blow-up formula (\cite{GS, FS0}), we need to show that the class $M$ extends to $X$ uniquely. First, we compute the dimension of the SW moduli space for $W= E(5) \# \CPb$ at the class $3f+E_1$:

\begin{eqnarray*}
d_W(3f+E_1) &=& \frac{\displaystyle{(3f+E_1)^2 -3\sigma(W)-2\chi(W)}} {\displaystyle{4}}\\
&=& \frac{\displaystyle{-1 +3(41) -2(61)}}{\displaystyle{4}}\\
&=&0.
\end{eqnarray*}
Let $\widetilde M$ be a basic class which is an extension of the class $3f+E_1$ to the manifold $X =( W \setminus \mathcal{Q}) \cup \mathcal{R}$. Since $b_2^+(X)>1$, $X$ is of simple type (\cite{GS}). Hence $\widetilde M$ must satisfy that $d_X(\widetilde M) = 0$, i.e.,
\begin{equation}
\widetilde M^2 = 3\sigma(X)+2\chi(X)
\label{6}
\end{equation}
Let us first compute 
\begin{equation}
(3f+E_1)|_{\mathcal{Q}} = ((3f+E_1) \cdot u_0)\gamma_0 + ((3f+E_1) \cdot u_1)\gamma_1 + ((3f+E_1) \cdot u_4)\gamma_4 = 3\gamma_0+\gamma_1 +\gamma_4
\label{M'}
\end{equation}
since $(3f+E_1) \cdot u_j = 0$ for $j = 2,3,5,6$.
Then, from the inverse matrix of the intersection form of $\mathcal{Q}$ given above, we find 
\begin{equation}
((3f+E_1)|_{\mathcal{Q}})^2 = (3\gamma_0+\gamma_1 +\gamma_4)^2 = -1315/261.
\label{7}
\end{equation}

Hence from equations \ref{6} and \ref{7} we have
\begin{eqnarray*}
0 &=& \widetilde M^2 - 3\sigma(X)- 2\chi(X)\\
&=& (3f+E_1)^2 - ((3f+E_1)|_{\mathcal{Q}})^2 + (\widetilde M|_{\mathcal{R}})^2 + 3 (36) - 2 (56)\\
&=& -1 + (1315/261) +(\widetilde M|_{\mathcal{R}})^2 - 4
\end{eqnarray*}

which gives
\begin{equation}
(\widetilde M|_{\mathcal{R}})^2 = -10/261. 
\label{8}
\end{equation}

From the intersection form of the filling $\mathcal{R}$ (Definiton \ref{QR}), we find its inverse:
\[1/261\begin{bmatrix}
-79&23\\
23&-10
\end{bmatrix}
\]

Let $r_1,r_2$ be the generators of the second homology of $\mathcal{R}$ with $r_1^2=-10, r_2^2=-79$, and $s_1,s_2$ be their Poincar\'{e} duals where $s_1^2=-79/261, s_2^2 = -10/261$. 
Let us write $\widetilde M|_{\mathcal{R}} = m(s_1)+n(s_2)$ for some $m,n \in \Z$. From Equation \ref{8}, 
\begin{eqnarray*}
(\widetilde M|_{\mathcal{R}})^2 = (m(s_1)+n(s_2))^2 &=& -10/261 \iff \\
m^2 (-79/261) +2mn (23/261)  + n^2(-10/261)&=& -10/261 \iff \\
79m^2 -46mn + 10n^2  = 10
\end{eqnarray*}
whose only integer (in fact rational) solutions are $m=0, n=\pm 1$. Hence 
\begin{equation}
\widetilde M|_{\mathcal{R}} = \pm s_2. 
\end{equation}
Now we need to show that exactly one of $(s_2)|_{\partial \mathcal{R}}$ or $(-s_2)|_{\partial \mathcal{R}}$ agrees with $(3f+E_1)|_{\partial(W \setminus \mathcal{Q})= \partial \mathcal{Q}}$. The restriction map is a homomorphism and we have that the restrictions of $s_2$ and $-s_2$ to $\partial \mathcal{R}$ are the same as the restrictions of $(3f+E_1)$ and $-(3f+E_1)$ to $\partial \mathcal{Q}$. From Equation \ref{M'} above, we have $M' := (3f+E_1)|_{\mathcal{Q}} = 3\gamma_0+\gamma_1 +\gamma_4$. By Lemma \ref{IMP}, to check whether $s_2|_{\partial\mathcal{R}}=-s_2|_{\partial\mathcal{R}}$, we need to check whether $M'$ and $-M'$ are in the same $2PD(H_2(\mathcal{Q},\Z))$ orbit. That is to say, we check whether $M' - (-M')=2M'$ is in the image of $2[\mathcal{Q}]$, i.e., whether the entries of the vector $[\mathcal{Q}]^{-1}M'$ are integers. Therefore we compute the product $[\mathcal{Q}]^{-1} [3,1,0,0,1,0,0]^T$ which is

$[-106/87, -193/261, -53/87, -27/29, -56/87, -212/261, -106/261]^T$. 

This shows that the restrictions of $s_2$ and $-s_2$ to $\partial \mathcal{R}$ are not the same. Moreover we know that at least one of $s_2$ or $-s_2$ has to be compatible with $3f+E_1$ on the boundary of $\mathcal{R}$, otherwise the manifold $X$ would not have any basic class. Thus, exactly one of $s_2$ or $-s_2$ is compatible with $3f+E_1$ on the boundary of $\mathcal{R}$. This shows that the class $3f+E_1$ extends uniquely to $X$ and $X$ has one basic class up to sign. Hence $X$ is minimal. 

Hence we proved the following theorem:

\begin{theorem}
There exists a simply connected, minimal, symplectic 4-manifold $X$ with an exotic smooth structure, and with one SW basic class up to sign, lying on the Noether line and obtained by the $(\mathcal{Q,R})$ star surgery.
\end{theorem}$\hfill\square$

\subsection{Second construction via the $(\mathcal{U,V})$-star surgery}
\label{firstbuild}
Now we construct a minimal, exotic 4-manifold $T$ via the $(\mathcal{U,V})$-star surgery, which is homeomorphic to $X$. 
First recall that in Lemma \ref{Uin} we have acquired the plumbing $\mathcal{U}$ symplectically embedded in $E(5)\#3\CPb$. Let

\begin{equation*}
T = ((E(5) \# 3\CPb) \setminus \mathcal{U}) \cup \mathcal{V} 
\end{equation*}
Then $\sigma(T) = \sigma(E(5) \# 3\CPb) - \sigma(\mathcal{U}) + \sigma(\mathcal{V})=-43+9-2 = -36$ and $e(T) = e(E(5) \# 3\CPb) - e(\mathcal{U}) + e(\mathcal{V}) =63-10+3 = 56.$
Thus, we have
\begin{equation}
\chi_h(T)=5 \;\;\;\text{and}\;\;\;  c_1^2(T) = 4= 2\chi_h-6
\end{equation}

This shows that $T$ is on the Noether line. From Van Kampen's theorem, we easily see that $T$ is simply connected as $\mathcal{V}$ is simply connected \cite{KS}. By Lemma \ref{exo}, we have that $T$ is an exotic copy of $9\CP \# 45 \CPb$. (In particular, $T$ is homeomorphic to $X$ which is constructed in the previous subsection). 

Now we prove that $T$ is minimal. The basic classes of $E(5) \# 3\CPb$ are $\pm f \pm E_1 \pm E_2 \pm E_3$ and $\pm 3f \pm E_1 \pm E_2 \pm E_3$, where $E_i$ are the Poincar\'{e} duals of the homology classes of the exceptional divisors coming from the blow ups. Hence in total there are 16 Seiberg-Witten basic classes up to sign. 

\begin{lemma}
Let $S$ be a basic class of $E(5) \# 3\CPb$, such that $S \neq \pm(3f +E_1 + E_2 + E_3)$, and let $\widetilde S$ be an extension of $S$ to $T$ where $\widetilde S|_{T \setminus \mathcal{V}} = S|_{(E(5) \# 3\CPb) \setminus \mathcal{U}}$. Then, the dimension of the SW moduli space satisfies that $d_{T}(\widetilde S)<0$ showing that $\widetilde S$ is not a basic class of $T$.
\end{lemma}

\begin{proof}
Proof is a direct computation. Let us take $S= f +E_1 + E_2 + E_3$. Let $\gamma_0,..., \gamma_8$ be the basis of $H^2(\mathcal{U},\Q)$ which is dual to $u_0,...,u_8$ (i.e., we have $\gamma_i(u_j)= \delta_{ij}$). (See Figure \ref{U} for the spheres $u_i$). Then 
\begin{eqnarray*}
S|_\mathcal{U} &= (S \cdot u_0)\gamma_0 +  (S \cdot u_3)\gamma_3 +  (S \cdot u_5)\gamma_5 +(S \cdot u_7)\gamma_7 +(S \cdot u_8)\gamma_8\\
&=  \gamma_0 + \gamma_3+\gamma_5+ +\gamma_7+\gamma_8
\end{eqnarray*}
From inverse of the intersection matrix $[\mathcal{U}]$ given in Figure \ref{UUinv} we find that
\begin{equation}
(S|_\mathcal{U})^2 = -821/289
\end{equation}
Now let us assume that there is a basic class $\widetilde S$ on $T$ as in the statement. Then $(\widetilde S|_{\mathcal{V}})^2 \leq 0$, since $\mathcal{V}$ is negative definite. Therefore the dimension of the SW moduli space satisfies the following:
\begin{eqnarray*}
d_T(\widetilde S) &=& \frac{\displaystyle{\widetilde S^2 -3\sigma(T)-2\chi(T)}} {\displaystyle{4}}\\
&=&\frac{\displaystyle{S^2 - (S|_{\mathcal{U}})^2+ (\widetilde S|_{\mathcal{V}})^2 -3\sigma(T)-2\chi(T)}} {\displaystyle{4}}\\
&=& \frac{\displaystyle{-3 +(821/289) +  (\widetilde S|_{\mathcal{V}})^2 +108-112}}{\displaystyle{4}}\\
&=& \frac{\displaystyle{-7 +(821/289)  +  (\widetilde S|_{\mathcal{V}})^2}}{\displaystyle{4}}\\
&<&0
\end{eqnarray*}
This shows that $\widetilde S$ is not a basic class of $T$.

Note that above we have $\big{|}(S|_{\mathcal{U}})^2\big{|} = 821/289 < 7$. Moreover, for every other class $S'$ as in the statement of the lemma, we find that $\big{|}(S'|_\mathcal{U})^2\big{|} <7$. Note that $(S')^2 = -3$, too. Therefore in each case we have $d_T(\widetilde S') <0$. 
\end{proof}

However, the top class $Y:=3f +E_1 + E_2 + E_3$ up to sign extends to $T$ as a basic class by Taubes' theorem. To prove that $T$ is minimal by the blow-up formula, we will show that $Y$ extends to $T$ uniquely. First we have $
d_{E(5) \# 3\CPb}(Y) =0$. Let $\widetilde Y$ be a basic class which is an extension of the class $Y$ to the manifold $T$. Since $b_2^+(T)>1$, $T$ is of simple type, hence
\begin{equation}
\widetilde Y^2 = 3\sigma(T)+2\chi(T)
\end{equation}
From the inverse matrix of the intersection form of $\mathcal{U}$ given above, we find 
\begin{equation}
(Y|_{\mathcal{U}})^2 = (3\gamma_0+\gamma_3 +\gamma_5+\gamma_7+\gamma_8)^2 = -2029/289.
\end{equation}
Hence we have
\begin{eqnarray*}
0 &=& \widetilde Y^2 - 3\sigma(T)- 2\chi(T)\\
&=& Y^2 - (Y|_{\mathcal{U}})^2 + (\widetilde Y|_{\mathcal{V}})^2 -4\\
&=& -3+(2029/289) +(\widetilde Y|_{\mathcal{V}})^2 - 4
\end{eqnarray*}

which gives
\begin{equation}
(\widetilde Y|_{\mathcal{V}})^2 = -6/289.
\label{Y2}
\end{equation}
Inverse of the intersection form of $\mathcal{V}$ is
\[
\begin{bmatrix}
(-49/1445) & (-1/289)\\
(-1/289) & (-6/289)
\end{bmatrix}
\]
(See Definition \ref{UV} for the intersection form of $\mathcal{V}$). Let $a_1,a_2$ be the generators of the second homology of $\mathcal{V}$ and $b_1,b_2$ be their duals where $b_1^2=-49/1445, b_2^2 = -6/289$. 
Let $\widetilde Y|_{\mathcal{V}} = m(b_1)+n(b_2)$ for some $m,n \in \Z$. From Equation \ref{Y2}
\begin{eqnarray*}
(\widetilde Y|_{\mathcal{V}})^2 = (m(b_1)+n(b_2))^2 &=& -6/289 \iff \\
m^2 (-49/1445) +2mn (-1/289)  + n^2(-6/289)&=& -6/289 \iff \\
49m^2 +10 mn + 30n^2  = 30
\end{eqnarray*}
whose only integer solutions are $m=0, n=\pm 1$. Hence 
\begin{equation}
\widetilde Y|_{\mathcal{V}} = \pm b_2. 
\end{equation}
Now we need to show that exactly one of $(b_2)|_{\partial \mathcal{V}}$ or $(-b_2)|_{\partial \mathcal{V}}$ agrees with $Y|_ {\partial \mathcal{U}}$. By Lemma \ref{IMP}, we compute the product $[\mathcal{U}]^{-1} [3,0,0,1,0,1,0,1,1]^T$ which is

\noindent $-1/289[407, 332, 257,182,302, 197, 302, 197,232]^T$. This shows that the restrictions of $b_2$ and $-b_2$ to $\partial \mathcal{R}$ are not the same and exactly one of $\pm b_2$ is compatible with $Y$ on the boundary of $\mathcal{V}$. This shows that the class $Y$ extends uniquely to $T$ and $T$ has one basic class up to sign. Hence $T$ is minimal. 

As a result, we have 

\begin{theorem}
There exists a simply connected, minimal, symplectic 4-manifold $T$ with an exotic smooth structure, and with one SW basic class up to sign, lying on the Noether line. $T$ is obtained by the $(\mathcal{U,V})$-star surgery and homeomorphic to the manifold $X$ constructed in Section~\ref{X}.
\end{theorem}$\hfill\square$


We note that starting with the configuration $\mathcal{U}$ in $E(5)\#3\CPb$, by two blow-downs and two symplectic resolutions we obtain the configuration $\mathcal{Q}$ in $E(5) \# \CPb$ as it can be directly seen from the proofs of lemmas \ref{Uin} and \ref{Qin}. 
However, we do not know if $X$ and $T$ are diffeomorphic to each other, it is an alluring problem.

\section{Constructions of simply connected, minimal, symplectic and exotic 4-manifolds between the Noether and half Noether lines}
In this section we construct simply connected, minimal, symplectic 4-manifolds with exotic smooth structures, lying between the Noether and half Noether lines and each with one basic class up to sign.
\label{4th}

\subsection{First construction by the $(\mathcal{K,L})$-star surgery}
\label{firstconst}
Our first construction is by the $(\mathcal{K,L})$-star surgery. We have constructed the plumbing $\mathcal{K}$ symplectically embedded in $E(6)$ in Lemma \ref{KSin}. Let us let
\begin{equation*}
Y = ((E(6) \setminus \mathcal{K}) \cup \mathcal{L}.
\end{equation*}
Then $\sigma(Y) = \sigma(E(6)) - \sigma(\mathcal{K}) + \sigma(\mathcal{L}) =-48+5-1 = -44$ and $e(Y)= e(E(6)) - e(\mathcal{K}) + e(\mathcal{L}) =72-6+2 = 68.$
Thus, 
\begin{equation}
\chi_h(Y)=6  \;\;\;\text{and}\;\;\; c_1^2(Y)= 4
\end{equation}
So, we have 
\begin{equation}
2\chi_h(Y) -6 > c_1^2(Y)> \chi_h(Y)-3
\end{equation}
which shows that $Y$ is in between the Noether and the half Noether lines.

The manifold $Y$ is simply connected. In fact, the generator of $\pi_1(\mathcal{L})$ can be isotoped into the boundary of $\mathcal{L}$ and it restricts to the boundary Seifert fibered space as a meridian of any of the -2 surgery curves in the plumbing diagram (Proposition 3.11 in \cite{KS}). On the other hand, by our construction, one of the spheres $u_j$ of $\mathcal{K}$ is a part of the $I_6$ fiber and the other transversally intersecting spheres are not cut out in the star surgery. Hence the meridian of $u_j$ bounds a disk in the complement of $\mathcal{K}$ which is contained in a sphere component of the $I_6$ fiber transversely intersecting $u_j$. That is to say, the generator of $\pi_1(\mathcal{L})$  is isotopic to the meridian of $u_j$ in the embedding which is homotopically trivial, hence $Y$ is simply connected (see also the proof of Theorem 5.22 in \cite{KS}). We have that $Y$ is an exotic copy of $11\CP \# 55\CPb$ by Lemma \ref{exo}.


Next we prove that $Y$ is minimal. Seiberg-Witten basic classes of $E(6)$ are $\pm 2f$ and $\pm 4f$ where $f \in H^2(Y,\Z)$ is the Poincar\'{e} dual of the homology class of the fiber. We need to determine which classes extend to $Y$. Let $P = 2f$ and  $\gamma_0,..., \gamma_4$ be the basis of $H^2(\mathcal{K},\Q)$ which is dual to $u_0,...,u_4$. (See the paragraph above Figure \ref{K} for the spheres $u_i$). Then \begin{eqnarray*}P|_\mathcal{K} &= (P \cdot u_0)\gamma_0 &=  2\gamma_0. \end{eqnarray*}From inverse of the intersection matrix $[\mathcal{K}]$ above we find that \begin{equation*}(P|_\mathcal{K})^2 = 4\gamma_0^2 = 4(-4/16) = -1. \end{equation*}

Now we assume that there is a basic class $\widetilde P$ on $Y$ such that $\widetilde P |_{Y-\mathcal{L}} = P |_{E(6)-\mathcal{K}}$. We have $(\widetilde P|_{\mathcal{L}})^2 \leq 0$ since the intersection form of $\mathcal{L}$ is negative definite.

Therefore the dimension of the SW moduli space:\begin{eqnarray*}d_Y(\widetilde P) &=& \frac{\displaystyle{\widetilde P^2 -3\sigma(Y)-2\chi(Y)}} {\displaystyle{4}}\\&=&\frac{\displaystyle{P^2 - (P|_{\mathcal{K}})^2+ (\widetilde P|_{\mathcal{L}})^2 -3\sigma(Y)-2\chi(Y)}} {\displaystyle{4}}\\&=& \frac{\displaystyle{0 +1 +  (\widetilde P|_{\mathcal{L}})^2 -4}}{\displaystyle{4}}\\&=& \frac{\displaystyle{(\widetilde P|_{\mathcal{L}})^2}-3}{\displaystyle{4}}\\&<&0\end{eqnarray*}
This contradicts our assumption that $\widetilde P$ is a basic class of $Y$. Therefore, the class $P$ does not descend to a basic class of $Y$.

On the other hand, by Taubes' theorem (\cite{Tau}) the top class $R = 4f$ descends to a basic class of $Y$ up to sign. By the blow-up formula (\cite{GS, FS0}) to conclude that $Y$ is minimal, we need to show that $R = 4f$ descends to $Y$ uniquely. First, we compute the dimension of the SW moduli space for $E(6)$ at the class $4f$:

\begin{eqnarray*}
d_{E(6)}(4f) &=& \frac{\displaystyle{(4f)^2 -3\sigma(E(6))-2\chi(E(6))}} {\displaystyle{4}}\\
&=& \frac{\displaystyle{0 -3(-48) -2(72)}}{\displaystyle{4}}\\
&=&0.
\end{eqnarray*}
Let $\widetilde{4f}$ be a basic class which is an extension of the class $4f$ to the manifold $Y =( E(6) \setminus \mathcal{K}) \cup \mathcal{L}$. Since $b_2^+(X)>1$, $Y$ is of simple type (\cite{GS}). Hence $\widetilde{4f}$ must satisfy that $d_Y(\widetilde{4f}) = 0$, i.e.,
\begin{equation}
\widetilde{4f}^2 = 3\sigma(Y)+2\chi(Y)
\end{equation}
Let us first compute 
\begin{equation}
((4f)|_{\mathcal{K}})^2 = (((4f) \cdot u_0)\gamma_0)^2 = (4 \gamma_0)^2 = 16 (-4/16)=-4
\end{equation}
since $(4f) \cdot u_j = 0$ for $j = 1, \cdots, 4$ and from the inverse of the intersection matrix of $\mathcal{K}$, we have $(\gamma_0)^2= -4/16$.

Hence we have
\begin{eqnarray*}
0 &=& \widetilde{4f}^2 - 3\sigma(Y)- 2\chi(Y)\\
&=& (4f)^2 - ((4f)|_{\mathcal{K}})^2 + (\widetilde{4f}|_{\mathcal{L}})^2 - 3 (-44) - 2 (68)\\
&=& 0 +4 +(\widetilde{4f}|_{\mathcal{L}})^2 - 4
\end{eqnarray*}

which gives
\begin{equation}
(\widetilde{4f}|_{\mathcal{L}})^2 = 0.
\label{4f}
\end{equation}

The intersection form of the filling $\mathcal{L}$ is $[-4]$ (see Definition \ref{KL}), and its inverse is $[-1/4]$. Hence, with Equation \ref{4f}, this gives that $\widetilde{4f}|_{\mathcal{L}}=0$. Therefore the class $4f$ extends to $Y$ uniquely. This shows that $Y$ is minimal.


Hence we have the following theorem:
\\
\begin{theorem}
There exists a simply connected, minimal, symplectic 4-manifold $Y$ with an exotic smooth structure, and with one SW basic class up to sign, lying in between the Noether and the half Noether lines, obtained by the $(\mathcal{K,L})$-star surgery.
\end{theorem}$\hfill\square$

\subsection{Second construction via the $(\mathcal{S}_2,\mathcal{T}_2)$ star surgery}

\label{last}
In this construction we will apply the $(\mathcal{S}_2,\mathcal{T}_2)$-star surgery to the elliptic surface $E(5)$. In Lemma \ref{KSin} we have shown that the plumbing $\mathcal{S}_2$ symplectically embeds in $E(5)$. Let us let
\begin{equation*}
Z = ((E(5) \setminus \mathcal{S}_2) \cup \mathcal{T}_2.
\end{equation*}
Then $\sigma(Z) = \sigma(E(5)) - \sigma(\mathcal{S}_2) + \sigma(\mathcal{T}_2)= -40 +5-2 = -37$ and $e(Z) = e(E(5)) - e(\mathcal{S}_2) + e(\mathcal{T}_2) = 60 - 6 +3 = 57$.
Thus, 
\begin{equation}
\chi_h(Z)=5 \;\;\; \text{and}\;\;\; c_1^2(Z)= 3.
\end{equation}
So, we have  
\begin{equation}
2\chi_h(Z) -6 > c_1^2(Z)> \chi_h(Z)-3
\end{equation}
which shows that $Z$ is in between the Noether and the half Noether lines. The manifold $Z$ is simply connected as in the previous example (see the proof of Lemma 5.2 in \cite{KS}). We also note that $b_2^+(Z) > 1$ and 
hence by Lemma \ref{exo} we conclude that $Z$ is an exotic copy of $9\CP \# 46 \CPb$.

Next we prove that $Z$ is minimal. Seiberg-Witten basic classes of $E(5)$ are $\pm f$ and $\pm 3f$. We will determine which classes extend to $Z$. Let $\gamma_0,..., \gamma_4$ be the basis of $H^2(\mathcal{S}_2,\Q)$ which is dual to $u_0,...,u_4$. (See the paragraph above Figure \ref{S2} for the spheres $u_i$). Then \begin{eqnarray*}f|_{\mathcal{S}_2} &= (f \cdot u_0)\gamma_0 &=  \gamma_0. \end{eqnarray*}From inverse of the intersection matrix $[\mathcal{S}_2]$ we find that \begin{equation*}(f|_{\mathcal{S}_2})^2 = \gamma_0^2 = -4/12. \end{equation*}

Now we assume that there is a basic class $\widetilde P$ on $Z$ such that $\widetilde P |_{Z-\mathcal{T}_2} = f |_{E(5)-\mathcal{S}_2}$. We have $(\widetilde P|_{\mathcal{T}_2})^2 \leq 0$ since the intersection form of $\mathcal{T}_2$ is negative definite. Therefore the dimension of the SW moduli space is\begin{eqnarray*}d_Z(\widetilde P) &=& \frac{\displaystyle{\widetilde P^2 -3\sigma(Z)-2\chi(Z)}} {\displaystyle{4}}\\&=&\frac{\displaystyle{f^2 - (f|_{\mathcal{S}_2})^2+ (\widetilde P|_{\mathcal{T}_2})^2 -3\sigma(Z)-2\chi(Z)}} {\displaystyle{4}}\\&=& \frac{\displaystyle{0 +4/12 +  (\widetilde P|_{\mathcal{T}_2})^2 -3}}{\displaystyle{4}}\\&=& \frac{\displaystyle{4/12+ (\widetilde P|_{\mathcal{T}_2})^2}-3}{\displaystyle{4}}\\&<&0.\end{eqnarray*}This contradicts our assumption that $\widetilde P$ is a basic class of $Z$. Therefore, the class $f$ does not descend to a basic class of $Z$. 

However by Taubes' theorem, we conclude that only the top class $3f$ descends to a basic class of $Z$, up to sign. Next we show that $3f$ descends to $Z$ uniquely. First note that we have
\begin{eqnarray*}d_{E(5)}(3f) &=& \frac{\displaystyle{(3f)^2 -3\sigma(E(5))-2\chi(E(5))}}{\displaystyle{4}}\\&=& \frac{\displaystyle{3(48)-2(72)}}{\displaystyle{4}}\\&=&\frac{\displaystyle{120-120}}{\displaystyle{4}} = 0.\end{eqnarray*}
Now, we assume that there is a basic class $\widetilde {3f}$ on $Z$ such that $\widetilde {3f} |_{Z-\mathcal{T}_2} = 3f |_{E(5)-\mathcal{S}_2}$. Since $b_2^+(Z)>1$, $Z$ is of simple type (\cite{GS}). Hence $\widetilde{3f}$ must satisfy that $d_Z(\widetilde{3f}) = 0$, i.e.,
\begin{equation}
\widetilde{3f}^2 = 3\sigma(Z)+2\chi(Z)
\end{equation}
We also have
\begin{equation*}(3f|_{\mathcal{S}_2})^2 = (3\gamma_0)^2= 9(-4/12)= -3.\end{equation*}

Hence we have the following
\begin{eqnarray*}
0 &=& \widetilde{3f}^2 - 3\sigma(Z)- 2\chi(Z)\\
&=& (3f)^2 - ((3f)|_{\mathcal{S}_2})^2 + (\widetilde{3f}|_{\mathcal{T}_2})^2 - 3 (-37) - 2 (57)\\
&=& 0 +3 +(\widetilde{3f}|_{\mathcal{T}_2})^2 - 3
\end{eqnarray*}

which gives
\begin{equation}
(\widetilde{3f}|_{\mathcal{T}_2})^2 = 0.
\label{3f}
\end{equation}
Now, inverse of the intersection form of $\mathcal{T}_2$ is 
\[
\begin{bmatrix}
    -1/4 &0\\       
    0&-1/3\\
\end{bmatrix}
\]
(See Definiton \ref{ST} for the intersection form of $\mathcal{T}_2$). Let $q_1,q_2$ be the generators of the second homology of $\mathcal{T}_2$ with $q_1^2=-4, q_2^2=-3$, and $t_1,t_2$ be their duals where $t_1^2=-1/4, t_2^2 = -1/3$. 
Let us write $\widetilde {3f}|_{\mathcal{T}_2} = m(t_1)+n(t_2)$ for some $m,n \in \Z$. 
\begin{eqnarray*}
(\widetilde {3f}|_{\mathcal{T}_2})^2 = (m(t_1)+n(t_2))^2 &=& 0\iff \\
m^2 (-1/4) + n^2(-1/3)&=& 0 
\end{eqnarray*}
whose only integer (in fact real) solutions are $m=n=0$. Hence
\begin{equation}
\widetilde {3f}|_{\mathcal{T}_2} = 0.
\end{equation}
This shows that the class $3f$ extends to $Z$ uniquely, and by the blow-up formula we have that $Z$ is minimal.


Hence we have proved the following:
\\
\begin{theorem}
There exists a simply connected, minimal, symplectic 4-manifold $Z$ with an exotic smooth structure, and with one SW basic class up to sign, lying in between the Noether and the half Noether lines, obtained by the $(\mathcal{S}_2,\mathcal{T}_2)$-star surgery.
\end{theorem}$\hfill\square$

\section{Applications of Remark \ref{kmgeneral}: A simply connected, minimal, symplectic and exotic 4-manifold above the Noether line}
\label{above}
Let us end by giving a construction of a simply connected, minimal, symplectic and exotic 4-manifold with one SW basic class up to sign, lying above the Noether line by using the $(\mathcal{U,V})$-star surgery and Remark \ref{kmgeneral}. We note that in the literature there are such exotic 4-manifolds with the same topological invariants. Below we give a different construction, without using knot surgery or mapping class groups. It would be interesting to compare our manifolds with the previously constructed ones to see whether they are diffeomorphic. 

We have shown that $\mathcal{U}$ symplectically embeds in $E(5) \# 2 \CPb$ (see Remark \ref{Uin2blowup}). By Remark \ref{kmgeneral}, we have, in particular, $\mathcal{U} \subset E(2) \# 5 \CPb$  and $\mathcal{U} \subset E(1) \# 6 \CPb$ symplectically. We apply the $(\mathcal{U,V})$-star surgery to first $E(2) \# 5 \CPb$. This gives a manifold $M$ above the Noether line as we see from the simple computation of the invariants. In fact, we have $\chi_h(M) =2, c_1^2(M) =2$ and by Van Kampen's theorem we show that $M$ is simply connected since $\mathcal{V}$ is simply connected \cite{KS}. Hence by Lemma \ref{exo} we have that $M$ is an exotic copy of $3\CP \# 17\CPb$. 

Now we show that $M$ is minimal. The proof goes parallel to the one in Section \ref{firstbuild}, but let us spell out the minor differences. Recall that $0 \in H^2(E(2),\Z)$ is the only basic class of $E(2)$. Then by the blow-up formula, $E(2)\#5\CPb$ has 32 basic classes: $\pm E_1 \pm \cdots \pm E_5$, i.e., 16 basic classes up to sign. In this construction we note that we use the symplectic embedding of the $(I_5,I_5)$ configuration in $E(2)$. Recall that we symplectically resolve the points $p,q$ and blow up the points $r,s$ as shown in Figure \ref{3}. Let us denote the exceptional divisors corresponding to $r,s$ by $E_1, E_2$, respectively. Next, as in Remark \ref{kmgeneral} we blow up the $-2$ section at 3 distinct points, let us call the exceptional spheres $E_3,E_4,E_5$. Hence we have the following intersections only:
\begin{eqnarray*}
E_1 \cdot u_3 = E_1 \cdot u_8 = 1\\
E_2 \cdot u_5 = E_2 \cdot u_7 = 1\\
E_3 \cdot u_0 = E_4 \cdot u_0 = E_5 \cdot u_0 = 1
\end{eqnarray*}
(See Figure \ref{U} for the spheres $u_i$). Let us prove
\begin{lemma}
Let $P$ be a basic class of $E(2) \# 5\CPb$, such that $P \neq \pm(E_1 + \cdots + E_5)$, and let $\widetilde P$ be an extension of $P$ to $M$ where $\widetilde P|_{M \setminus \mathcal{V}} = P|_{(E(2) \# 5\CPb) \setminus \mathcal{U}}$. Then, the dimension of the SW moduli space satisfies that $d_{M}(\widetilde P)<0$ showing that $\widetilde P$ is not a basic class of $M$.
\end{lemma}
\begin{proof}
Proof is a direct computation. Let us take $P = E_1 + E_2 + E_3+E_4 -E_5$. From inverse of the intersection matrix $[\mathcal{U}]$ given in Figure \ref{UUinv} we find that
\begin{eqnarray*}
(P|_\mathcal{U})^2 &= (\gamma_3 + \gamma_8 + \gamma_5 + \gamma_7+ \gamma_0 + \gamma_0 - \gamma_0)^2\\
&=  (\gamma_0 + \gamma_3+\gamma_5+ +\gamma_7+\gamma_8)^2\\
&=-821/289
\end{eqnarray*}
where $\gamma_i$ are the basis elements of $H^2(\mathcal{U},\Q)$ dual to $u_i$. Now let us assume that there is a basic class $\widetilde P$ on $M$ as in the statement. Then $(\widetilde P|_{\mathcal{V}})^2 \leq 0$, since $\mathcal{V}$ is negative definite. Therefore the dimension of the SW moduli space satisfies the following:
\begin{eqnarray*}
d_M(\widetilde P) &=& \frac{\displaystyle{\widetilde P^2 -3\sigma(M)-2\chi(M)}} {\displaystyle{4}}\\
&=&\frac{\displaystyle{P^2 - (P|_{\mathcal{U}})^2+ (\widetilde P|_{\mathcal{V}})^2 -3(-14)-2(22)}} {\displaystyle{4}}\\
&=& \frac{\displaystyle{-5 +(821/289) +  (\widetilde P|_{\mathcal{V}})^2 -2}}{\displaystyle{4}}\\
&=& \frac{\displaystyle{-7 +(821/289)  +  (\widetilde P|_{\mathcal{V}})^2}}{\displaystyle{4}}\\
&<&0
\end{eqnarray*}
This shows that $\widetilde P$ is not a basic class of $M$.

Moreover, for every other class $P'$ as in the statement of the lemma, we find that $\big{|}(P'|_\mathcal{U})^2\big{|} <7$, therefore in each case we have $d_M(\widetilde P') <0$. 
\end{proof}
However, the top class $K:=E_1 + \cdots+ E_5$ up to sign extends to $M$ as a basic class by Taubes' theorem. To prove minimality, we will show that $K$ extends to $M$ uniquely. First we have $
d_{E(2) \# 5\CPb}(K) =0$. Let $\widetilde K$ be a basic class which is an extension of the class $K$ to the manifold $M$ which is of simple type. Hence
\begin{equation}
\widetilde K^2 = 3\sigma(M)+2\chi(M)
\end{equation}
From the inverse matrix of the intersection form of $\mathcal{U}$ given above, we find 
\begin{eqnarray*}
(K|_{\mathcal{U}})^2 &=(\gamma_3 + \gamma_8 + \gamma_5 + \gamma_7+ \gamma_0 + \gamma_0 +\gamma_0)^2\\
&=(3\gamma_0+\gamma_3 +\gamma_5+\gamma_7+\gamma_8)^2 = -2029/289.
\end{eqnarray*}
Hence we have
\begin{eqnarray*}
0 &=& \widetilde K^2 - 3\sigma(M)- 2\chi(M)\\
&=& -5+(2029/289) +(\widetilde K|_{\mathcal{V}})^2 - 2
\end{eqnarray*}
which gives
\begin{equation}
(\widetilde K|_{\mathcal{V}})^2 = -6/289
\end{equation}
as in Equation \ref{Y2} of Section \ref{firstbuild}. We note that rest of the proof is exactly the same as in Section \ref{firstbuild}, after Equation \ref{Y2}. Hence we have 
\begin{theorem}
There exists a simply connected, minimal, symplectic 4-manifold $M$ with an exotic smooth structure, and with one SW basic class up to sign, lying above the Noether line and obtained by the $(\mathcal{U,V})$-star surgery.
\end{theorem}$\hfill\square$

Let us also ask the following: 
\begin{question}
Can we build exotic 4-manifolds, via star surgeries, that are arbitrarily close to the BMY-line $c_1^2 =9\chi_h$?
\end{question}

Lastly, let us note that if we apply the $(\mathcal{U,V})$-star surgery to $E(1) \# 6 \CPb$, we see that the resulting symplectic manifold has $\chi_h = 1$, and $c_1^2=1$, so it is homeomorphic to $\CP \# 8 \CPb$. Exoticness can be shown from the symplectic Kodaira dimension, but since this is $b_2^+=1$ case, the proof of minimality is longer. However, this does not improve the results in \cite{KS} and we also use the same technique; a star surgery. Therefore, we will not pursue these computations here. Note that to obtain smaller exotic manifolds, one needs to consider additional fishtail and cusp fibers in the starting manifolds as in \cite{KS}. But in this paper our main interest is manifolds on and below the Noether line, and we have only worked with $I_n$ fibers.


\bibliography{References}
\bibliographystyle{spmpsci}

%
%

\end{document}